\documentclass[abstractpage,publicabstractpage,ackpage,copyrightpage,phd]{uithesis}
\usepackage{amsfonts}
\usepackage{amsmath}
\usepackage{amssymb}
\usepackage{amsthm} 
\usepackage{array}
\usepackage{float}
\usepackage[toc,page,titletoc]{appendix}
\usepackage{longtable}
\usepackage{graphicx}
\usepackage[font=singlespacing]{caption}
\usepackage[normalem]{ulem}
\usepackage{nicefrac}
\usepackage{units}
\usepackage{rotating}
\usepackage{afterpage}
\usepackage{tikz}
\usetikzlibrary{arrows,positioning}

\usepackage{etoolbox}
\apptocmd{\thebibliography}{\interlinepenalty 10000\relax}{}{}

\usepackage{enumitem}
\setlist{noitemsep,topsep=1pt, partopsep=4pt, parsep=2pt}
\setenumerate{noitemsep,topsep=1pt, partopsep=4pt, parsep=2pt}

\allowdisplaybreaks[2]

\newtheorem{thm}{Theorem}[chapter]
\newtheorem{lem}[thm]{Lemma}
\newtheorem{cor}[thm]{Corollary}
\newtheorem{prop}[thm]{Proposition}

\theoremstyle{definition}
\newtheorem{definition}{Definition}[chapter]

\theoremstyle{remark}

\newtheorem*{remark}{Remark}

\newcommand{\lyxmathsym}[1]{\ifmmode\begingroup\def\b@ld{bold}
  \text{\ifx\math@version\b@ld\bfseries\fi#1}\endgroup\else#1\fi}
\providecommand{\LyX}{L\kern-.1667em\lower.25em\hbox{Y}\kern-.125emX\@}

\def\sideremark#1
{
\ifvmode\leavevmode\fi\vadjust
  {\vbox to0pt
    {\vss\hbox to0pt
      {\hskip\hsize\hskip1em\vbox
        {\hsize2cm
          \scriptsize
          \raggedright\pretolerance10000\noindent#1\hfill
        }
        \hss
      }
      \vbox to8pt{\vfil}\vss
    }
  }
}

\title{Classification of tensor decompositions for II$_1$ factors}

\author{Wanchalerm Sucpikarnon}
\dept{Mathematics}
\setboolean{multipleSupervisors}{false}
\advisor{Ionut Chifan, Associate Professor} 
\memberOne{Ionut Chifan}
\memberTwo{Raul Curto}
\memberThree{Palle Jorgensen}
\memberFour{Surjit Khurana}
\memberFive{Victor Camillo}
\submitdate{December 2019}
\copyrightyear{2019}{}

\ackfile{thesisAck}
\abstractfile{thesisAbstract}
\publicabstractfile{thesisAbstractpublic}

\begin{document}

\frontmatter

\chapter{INTRODUCTION} \label{chapter:chapterlabelone}

An important step towards understanding the structure of II$_1$ factors is the study of their tensor product decompositions.  A factor is called \emph{prime} if it cannot be decomposed as a tensor product of diffuse factors. Using this notion of $*$-orthogonal von Neumann algebras, S. Popa was able to show in \cite{Po83} that the (non-separable) II$_1$ factor $L(\mathbb F_S)$ arising from the free group $\mathbb F_S$ with uncountably many generators $S$ is prime. More than a decade later, using Voiculescu's influential free probability theory, Ge managed to prove the same result about the free group factors $L(\mathbb F_n)$ with countably many generators, $n\geq 2$ \cite{Ge98}.  Using a completely different perspective based on $C^*$-techniques, Ozawa obtained a far-reaching generalization of this by showing that for every icc hyperbolic group $\Gamma$ the corresponding factor $L(\Gamma)$ is in fact \emph{solid} (for every diffuse $A\subset L(\Gamma)$ von Neumann subalgebra, its relative commutant $A'\cap L(\Gamma)$ is amenable) \cite{Oz03}.  Developing a new approach rooted in the study of closable derivations, Peterson showed primeness of $L(\Gamma)$, whenever $\Gamma$ is any nonamenable icc group with positive first Betti number \cite{Pe06}. Within the powerful framework of his deformation/rigidity theory Popa discovered a new proof of solidity of the free group factors \cite{Po06}. These methods laid out the foundations of a rich subsequent activity regarding the study of primeness and other structural aspects of II$_1$ factors \cite{Oz04, CH08, CI08, Si10, Fi10,CS11,CSU11, SW11, HV12,Bo12, BHR12, DI12, CKP14, Is14, HI15, Ho15, DHI16, Is16}.      

\section{Statements of main results} The techniques introduced in the deformation/rigidity framework also opened up a whole array of new possibilities towards understanding novel aspects in the classification of tensor product decompositions of factors. For example, motivated in part by the results in \cite{CdSS15}, Drimbe, Hoff and Ioana have discovered in \cite{DHI16} a new classification result regarding the study of tensor product decompositions of II$_1$ factors. Precisely, whenever $\Gamma$ is an icc group that is measure equivalent to a direct product of non-elementary hyperbolic groups then \emph{all} possible tensor product decompositions of the corresponding II$_1$ factor $L(\Gamma)$ can arise \emph{only} from the canonical direct product decompositions of the underlying group $\Gamma$. Pant and de Santiago showed the same result holds when $\Gamma$ is a poly-hyperbolic group with non-amenable factors in its composition series \cite{dSP17}. In this dissertation we make new progress in this direction by introducing several new and fairly large classes of groups for which this tensor product rigidity phenomenon still holds. This include many new families of groups that were not previously investigated in this framework such as amalgamated free products and McDuff groups. Our results also improve significantly upon a series of previous results on primeness and unique prime factorisations including \cite{CH08,SW11}. Below we briefly describe these results also placing them in a context and explaining their importance and the methods involved. 

Basic properties in Bass-Serre theory of groups show that the only way an amalgam $\Gamma_1\ast_{\Sigma}\Gamma_2$ could decompose as a direct product is through its core $\Sigma$. Precisely, if $\Gamma_1\ast_{\Sigma} \Gamma_2 =\Lambda_1\times \Lambda_2$ then there is a permutation $s$ of $\{1,2\}$ so that $\Lambda_{s(1)}<\Sigma$. This further gives $\Sigma= \Lambda_{s(1)}\times \Sigma_0$, $\Gamma_1= \Lambda_{s(1)} \times \Gamma^0_1$, $\Gamma_2= \Lambda_{s(1)} \times \Gamma^0_2$ for some groups $\Sigma_0<\Gamma^0_1,\Gamma^0_2$  and hence $\Lambda_{s(2)} =\Gamma^0_1\ast_{\Sigma_0} \Gamma^0_2$. An interesting question is to investigate situations when this basic group theoretic aspect could be upgraded to the von Neumann algebraic setting. It is known this fails in general since there are examples of product indecomposable icc amalgams whose corresponding factors are McDuff and hence decomposable as tensor products. However, under certain indecomposability assumptions on the core algebra, we are able to provide a positive answer to our question.

\begin{thm}
\label{tensordecompampf} Let $\Gamma=\Gamma_1\ast_{\Sigma}\Gamma_2$ be an icc group such that $[\Gamma_1:\Sigma]\geq 2$ and $[\Gamma_2:\Sigma]\geq 3$. Assume that $\Sigma$ is finite-by-icc and any corner of $L(\Sigma)$ is virtually prime. Suppose that $L(\Gamma)=M_1\bar\otimes M_2$, for diffuse $M_i$'s. Then there exist decompositions $\Sigma=\Omega \times \Sigma_0$ with $\Sigma_0$ finite, $\Gamma_1= \Omega \times \Gamma^0_1$, $\Gamma_2= \Omega \times \Gamma^0_2$, for some groups $\Sigma_0< \Gamma_1^0,\Gamma^0_2$, and hence $\Gamma =\Omega \times (\Gamma^0_1\ast_{\Sigma_0} \Gamma^0_2)$. Moreover, there is a unitary $u\in L(\Gamma)$, $t>0$, and a permutation $s$ of $\{1,2\}$ such that 
\begin{equation*}
M_{s(1)}= u L(\Omega)^t u^*\quad{ and }\quad  M_{s(2)} = u L(\Gamma_1^0\ast_{\Sigma_0} \Gamma_2^0)^{1/t} u^*.
\end{equation*}
\end{thm}
In particular this result provides many new examples of prime group factors and factors that satisfies Ozawa-Popa's unique prime decomposition property. This includes factors associated with simple groups such as Burger-Mozes groups which is a premiere in the subject.

In \cite{Po07} Popa was able to establish primeness for all factors $L(\Gamma)$ associated with non-canonical wreath product groups $\Gamma=A\wr G$ where $A$ is amenable and $\Gamma$ is non-amenable. Using the deformation techniques from \cite{CPS12} Sizemore and Winchester were able to extend this result by establishing various unique tensor decomposition properties from von Neumann algebras arising from direct products of such groups. In this dissertation we extend this even further by showing that all tensor decompositions of such factors are in fact parametrized by the canonical direct product decompositions of the underlying group. Specifically, for product of groups in the class $\mathcal{WR}$ (see section 6.3.2 for the definition) we have the following result

\begin{thm}
Let $\Gamma_1,\Gamma_2, \ldots, \Gamma_n\in \mathcal{WR}$ and let $\Gamma=\Gamma_1\times\Gamma_2\times \cdots \times \Gamma_n$. Consider the corresponding von Neumann algebra $M=L(\Gamma)$ and let $P_1, P_2$ be non-amenable II$_1$ factors such that $M=P_1\bar\otimes P_2$. Then there exist a scalar $t>0$ and a partition $I_1\sqcup I_2=\{1,2,\ldots,n\}$ such that
\begin{equation*}
L(\Gamma_{I_1})\cong P_1^t \quad \text{and} \quad L(\Gamma_{I_2})\cong P_2^{1/t}.
\end{equation*}
\end{thm}

In the celebrated work \cite{Mc69} McDuff introduced an (uncountable) family of groups that give rise to non-isomorphic II$_1$ factors, thus solving a long standing open problem at the time. Her construction of these groups was quite involved being essentially based on the iteration of the so-called $T_0$ and $T_1$ group functors. These functors are in part inspired by the earlier work of Dixmier and Lance \cite{DL69} which in turn go back to the pioneering work of Murray and von Neumann \cite{MvN43}.

Let $\Gamma$  be a group. For $i\geq 1,$ let $\Gamma_i$ be isomorphic copies of $\Gamma$
 and $\Lambda_i$ be isomorphic to $\mathbb{Z}$. Define $\tilde\Gamma=\bigoplus_{i\geq 1} \Gamma_i$ and let $\mathfrak S_{\infty}$ be the group of finite permutations of the positive integers $\mathbb N$. Consider the semidirect product $\tilde\Gamma\rtimes \mathfrak S_{\infty}$ associated to the natural action of $\mathfrak S_{\infty}$ on $\tilde\Gamma$ which permutes the copies of $\Gamma$. Following \cite{Mc69} we define
\begin{itemize}
\item $T_0(\Gamma)$ = the group generated by $\tilde\Gamma$ and $\Lambda_i, i\geq 1$ with the only relation that $\Gamma_i$ and $\Lambda_j$ commutes for $i\geq j\geq 1$.
\item  $T_1(\Gamma)$ = the group generated by $\tilde\Gamma\rtimes \mathfrak S_{\infty}$ and $\Lambda_i, i\geq 1$ with the only relation that $\Gamma_i$ and $\Lambda_j$ commute for $i\geq j\geq 1$.
\end{itemize}

From definitions it is evident that $T_i(\Gamma)$ give rise to II$_1$ factors $L(T_i(\Gamma))$ that have an abundence of assymptoticaly central sequences and hence by \cite{Mc69} they admit many tensor product decompositions by the hyperfinite factor, i.e. $L(T_i(\Gamma))\cong L(T_i(\Gamma))\bar \otimes \mathcal R$. However, besides this classic result, virtually noting is known towards describing the other possible tensor decompositions of these factors. In this thesis we completely answer this question by showing that in fact these are \emph{all} the possible tensor decompositions of these factors.

\begin{thm}
Fix $\Gamma$  a non-amenable group and let $\alpha\in\{0,1\}$. If $L(T_{\alpha}(\Gamma))=P_1\bar\otimes P_2$ then either $P_1$  or $P_2$ is isomorphic to the hyperfinite II$_1$ factor.

\end{thm}

All the aforementioned results are obtained through the developments of several new technical innovations in the deformation/rigidity technology. These new methods are highlighted in the chapter 6 of this thesis which also contains the bulk of the results. Particularly important in most of the proofs is the notion of spatial commensurability for von Neumann subalgebras introduced in the section 6.2 as well as the assymptotic analysis on bimodules and clustering von Neumann sugalgebras presented in the proof of theorem 6.14. These new methods shed new light in the study of tensor decomposition aspects present excellent potential to tackle more difficult groups that will be investigated in the future. In this dissertation, Theorem 1.1 is from the previous work with R. de Santiago while Theorem 1.2 and Theorem 1.3 are the results of a collaboration with I. Chifan.


\chapter{Von Neumann Algebras} \label{chapter:chapterlabelone}

\section{Introduction}

Let $H$ be a Hilbert space and $\mathcal{B}(H)$ be the space of bounded linear operators on $H$.
Recall that $\mathcal{B}(H)$ is a Banach space with the operator norm $\|\cdot\|_{\infty}$.

We define the convergences on $\mathcal{H}$ as the following:
\begin{itemize}
\item The \textit{uniform topology} is a topology defined by the operator norm, i.e,
\begin{equation*}
\text{$x_n\rightarrow x$ \,uniformly \quad if and only if \quad $\|x_n-x\|_{\infty}\rightarrow 0$.}
\end{equation*}
\item The \textit{strong operator topology} (SOT) is a topology generated by the family of semi-norm $\|x\xi\|$ for all \,$x\in \mathcal{B}(H)$ and $\xi\in H$, i.e. 
\begin{equation*}
\text{$x_n\rightarrow x$ \,SOT \quad if and only if \quad $\|(x_n-x)\xi\|\rightarrow 0$ \,\,for all\,\,$\xi\in H.$}
\end{equation*}
\item The \textit{weak operator topology} (WOT) is a topology generated by the family of semi-norm $|\langle x \xi, \zeta\rangle|$ for all $x\in \mathcal{B}(H)$ and $\xi,\zeta\in H$, i.e. 
\begin{equation*}
\text{$x_n\rightarrow x$ \,WOT \quad if and only if\quad $|\langle (x_n-x) \xi, \zeta\rangle|\rightarrow 0$ \,\,for all\,\, $\xi, \zeta\in H.$}
\end{equation*}

\end{itemize}
Note that the topologies on $\mathcal{B}(H)$ can be compared as the following relation:
\begin{center}
WOT$\quad\prec\quad$ SOT$\quad\prec\quad$  uniform.
\end{center}

\begin{definition}
A \textbf{von Neumann algebra} is a $*$-subalgebra of $\mathcal{B}(H)$ containing the unit $1$ and being closed in the weak topology.
\end{definition}

\begin{definition}
Let $B\subset \mathcal{B}(H)$, the commutant of $B$ is defined by
\begin{equation*}
B'=\{x\in \mathcal{B}(H)\,|\, xy=yx \,\,\,\text{for all}\,\, y\in B\}.
\end{equation*} 
\end{definition}

\begin{thm}[Double commutant Theorem]
Let $A$ be a $*$-subalgebra of $\mathcal{B}(H)$ containing the unit $1$. 
\begin{equation*}
A''=\overline{A}^{WOT}=\overline{A}^{SOT}.
\end{equation*}
In particular, $A$ is a von Neumann algebra if and only if $A=A''$.
\end{thm}

In general, by Double commutant Theorem, for any subset $S\subset \mathcal{B}(H)$ we call $(S\cup S^*)''$ a von Neumann algebra generated by $S$. Moreover for subsets $S_1, S_2\subset \mathcal{B}(H)$, we write
$S_1\vee S_2$
as the von Neumann algebra generated by $S_1$ and $S_2$.

If $S\subset M$ are von Neumann algebras, then
\[
P'\cap M=\{x\in M\,|\, xp=px\,\,\text{for}\,\,p\in P\}
\]
is called the \textbf{relative commutant} of $P$ in $M$.

\begin{thm}[Kaplansky Density Theorem]
Let $N\subset \mathcal{B}(H)$ be a von Neumann algebra and $A$ be a strongly $*$-subalgebra, not assumed to be unital.
\begin{enumerate}
\item[(i)] If $x\in N$, then there exists a net $(x_{\alpha})$ from $A$ converging $*$-strongly to $x$ and satisfying $\|x_{\alpha}\|\leq \|x\|$ for all $\alpha.$  
\item[(ii)] If $x\in N$ is a self-adjoint then the net in $(i)$ may be chosen with the additional property that each $x_{\alpha}$ is self-adjoint.
\item[(iii)] If $u\in N$ is a unitary and $A$ is a unital C$^*$-algebra, then there is a net $(u_{\alpha})$ of $unitaries$ from $A$ converging $*$-strongly to $u$.
\end{enumerate}
\end{thm}

\begin{definition} 
Let $A$ be a subset of $\mathcal{B}(H)$. Then the \textbf{center} of $A$ is defined by
\begin{equation*}
\mathcal{Z}(A)=A \cap A'.
\end{equation*}
\end{definition}

\begin{definition}
A von Neumann algebra $M$ is called \textbf{factor} if it has the trivial center, i.e. $\mathcal{Z}(M)=\mathbb{C}1$.
\end{definition}

\begin{definition}
A von Neumann algebra $M$ is \textbf{finite} if it has a faithful normal tracial state $\tau:M\rightarrow \mathbb{C}$ satisfying:
\begin{itemize}
\item $\tau$ is a positive linear functional with $\tau(1)=1$; 
\item $\tau$ is faithful, i.e. if $\tau(x^*x)=0$ then $x=0$;
\item $\tau$ is normal, i.e. weakly continuous on $(M)_1$, the unit ball of $M$ with respect to the uniform norm $\|\cdot\|_{\infty}$;
\item $\tau$ is trace, i.e. $\tau(xy)=\tau(yx)$ for all $x,y\in M$.
\end{itemize}
If $M$ is an infinite dimensional finite von Neumann algebra, then $M$ is called an von Neumann algebra of \textbf{type II$_1$}.
\end{definition}

Let $M\in \mathcal{B}(H)$ be a von Neumamm algebra on a Hilbert space $H$ and $p$ a projection in $H$. Then
\begin{equation*}
pMp=\{pxp\,|\, x\in M\}
\end{equation*}
is a von Neumann algebra in $\mathcal{B}(pH)$. One says that $pMp$ is  a corner of $M$.

\begin{definition}
A von Neumann algebra M is \textbf{diffuse} if there are no nonzero minimal projecton or an atom in M. 
Recall that a nonzero projection $p\in M$ is said to be \textbf{minimal} if the corner $pMp=\mathbb{C}p.$ 
\end{definition}

\begin{definition}
We say that two von Neumann algebra $M_1$  and $M_2$ are \textbf{isomorphic} if there exists a bijection $*$-homomorphism (called an isomorphism) between $M_1$ and $M_2$, and denoted by $M_1\cong M_2$.
\end{definition}

\begin{remark}
If a finite factor $(M,\tau)$ has a minimal projection, then $M\cong M_n(\mathbb{C})$ for some $n$.
A finite factor M is diffuse if and only if it is infinite dimensional. M is then a type II$_1$ factor.
\end{remark}

\section{Group von Neumann algebra}
Let $\Gamma$ be a discrete group and $\ell^2(\Gamma)$ be the space of square summable sequences over $\Gamma$ which has a natural orthonormal basis 
$\{\delta_h\,|\, h\in \Gamma\},$
where $\delta_h$ is a Kronecker delta. Consider the left regular representation of $\Gamma$ on $\ell^2(\Gamma)$ defined by 
$\lambda_{g}(\delta_h)=\delta_{gh} \quad \text{for}\,\,g,h\in \Gamma. $
The von Neumann algebra $L(\Gamma)$  is a von Neumann algebra generated by the set $\{\lambda_g\,|\, g\in \Gamma\}$. By the bicommutant theorem 
\begin{equation}
L(\Gamma)={\operatorname{span}\{\lambda_g\,|\, g\in \Gamma\}}''\cong\overline{\mathbb{C}[\Gamma]}^{\text{SOT}}
\end{equation}

In addition, recall that for $f, f_0\in \ell^2(\Gamma)$, the convolution product
\[
L_f(f_0)=f*f_1
\]
defined by 
\[
(f*f_0)(t)=\sum_{s\in \Gamma}f(s)f_0(s^{-1}t).
\]
By Cauchy-Schwarz inequality, we have $\|f*f_0\|_{\infty}\leq\|f\|_2\|f_0\|_2$ and it follows that $f*f_0\in \ell^{\infty}(\Gamma)$.
We say that $f$ is a \textit{left} convolver for $\Gamma$ if $f*f_0\in \ell^2(\Gamma)$ for every $f_0\in \ell^2(\Gamma)$. We denote $LC(\Gamma)$ the space of all left convolvers for $\Gamma$. Since $LC(\Gamma)$ commutes with $\rho(\Gamma)$, so we can view $LC(\Gamma)$ as a subspace of $\rho(\Gamma)'$ in $\mathcal{B}(\ell^2(\Gamma))$. Plus, it is easy to check that $\lambda(\Gamma)$ is contained in $LC(\Gamma)$. Therefore, $L(\Gamma)\subset LC(\Gamma)$. Similarly, we can introduce $RC(\Gamma)$ as the space of all right convolvers for $\Gamma$. As a consequence, one can show that
\begin{center}
 $LC(\Gamma)=L(\Gamma)=RC(\Gamma)'$ \,\,\,and \,\,\,$RC(\Gamma)=R(\Gamma)=LC(\Gamma)'$
\end{center}

The von Neumann algebra $L(\Gamma)$ is the (left) \textbf{group von Neumann algebra} of $\Gamma$ and $R(\Gamma)$ is the right group von Neumann algebra of $\Gamma$. Note that since the left and right-regular representations are equivalent it follows that $L(\Gamma)\cong R(\Gamma)$.

Moreover for $x\in L(\Gamma)$, 
\begin{equation*}
\tau(x)=\langle x\delta_e, \delta_e\rangle
\end{equation*}
 defines a normal faithful trace on $L(\Gamma)$. In particular, $L(\Gamma)$ is a finite von Neumann algebra.

If $x=\sum_{g\in\Gamma}\alpha_g\delta_g\in \ell^2(\Gamma)$ is a left-convolver, then we will often also write $x$ or $\sum_{g\in\Gamma}\alpha_gu_g$ to denote the operator $L_x\in LC(\Gamma).$ (Instead of $\delta_g$ we use $u_g$ to emphasize that $u_g$ is a unitary operator.) And we call the set $\{\alpha_g\}_{g\in \Gamma}$ the \textbf{Fourier coefficients} of $x$. Thus writing $x=\sum_{g\in\Gamma}\alpha_gu_g$ should be considered as an abbreviation for writing $L_{x}=L_{\sum_{g\in\Gamma}\alpha_g\delta_g}$.

\begin{thm}\cite{MvN43} 
Let $\Gamma$ be a discrete group  
A group von Neumann algebra $L(\Gamma)$ is a factor if and only if $\Gamma$ is an infinite conjugacy class (icc) group, i.e. each conjugacy class of non-trivial elements in $\Gamma$ is an infinite set.
\end{thm}

\begin{proof}
Suppose $h\in \Gamma\setminus \{e\}$ and the conjugacy class $h^{\Gamma}=\{ghg \,|\, g\in \Gamma\}$ is finite. Then $x=\sum_{k\in h^{\Gamma}}u_k\not\in \mathbb{C}$ and $x\in\{u_g\}'_{g\in \Gamma}\cap L(\Gamma)=\mathcal{Z}(L(\Gamma))$.

Conversely, suppose $\Gamma$ is icc and $x=\sum_{g\in\Gamma}\alpha_{g}u_g\in \mathcal{Z}(L(\Gamma))\setminus \mathbb{C},$ then for all $h\in \Gamma$ we have 
\begin{equation*}
x=u_hxu_h^*=\sum_{g\in \Gamma}\alpha_{h}u_{hgh^{-1}}=\sum_{g\in \Gamma} \alpha_{h^{-1}gh}u_g.
\end{equation*}
Thus the Fourier coefficients for $x$ are constant on conjugacy classes and since $\sum_{g\in \Gamma} |\alpha_{g}|<\infty$. Then we have $\alpha_g=0$ for all $g\neq e$ and hence $x=\alpha_{e}\in \mathbb{C}$.
 \end{proof}

The followings are examples of countable icc groups
\begin{itemize}
\item $\mathbb{F}_n, n\geq 2$ the free group on $n$ generators.
\item $\mathfrak{S}_{\infty}=\cup_{n=1}^{\infty}\mathfrak{S}_{n}$ the group of finite permutations on $\mathbb{N}$
\item wreath products $H\wr_I \Gamma:=(\oplus_I H)\rtimes \Gamma$ where $H, \Gamma$ are countably infinite and $\Gamma\curvearrowright I$ with infinite orbits.
\item icc property is closed under products.
\item amalgamated free products $\Gamma=\Gamma_1*_{\Sigma}\Gamma_2$ where $[\Gamma_1:\Sigma]\geq 2, [\Gamma_2:\Sigma]\geq 3$ and $|\Sigma\cap g\Sigma g^{-1}|<\infty$.
\end{itemize}

\section{Group measure space}
Let $\Gamma\curvearrowright (X,\mu)$ be a probability measure preserving (p.m.p.) action of $\Gamma$ on a probability space $(X,\mu)$. Recall that $L^{\infty}(X,\mu)$ acts naturally by multiplication on $L^2(X,\mu)$. Let $\sigma:\Gamma\curvearrowright L^2(X,\mu)$ be an action of $\Gamma$ on $L^2(X,\mu)$ defined by
\begin{equation*}
\sigma_g(f)(x)=f(g^{-1}x)\quad \text{for all}\quad g\in\Gamma, x\in (X,\mu).
\end{equation*}
Define the space
\begin{equation*}
A[\Gamma]:=\bigg\{   \sum_{g\in\Gamma} a_g g \,\big|\, a_g\in L^{\infty}(X,\mu)\,\,\text{and}\,\, a_g=0\,\,\text{for all $g\in\Gamma$ but finite}\bigg\}.
\end{equation*}
The product is defined by
$(a_1g)(a_2h)=a_1\sigma_g(a_2) gh$
and the involution by
$(ag)^*=\sigma_{g^{-1}}(a^*)g^{-1}$ where $a^*=\bar a.$ To avoid any confusion, we write $u_g$ instead of $1_{L^{\infty}(X,\mu)} g$ According to the first step, it follows that $A[\Gamma]$ is a $*$-algebra of operators acting on the Hilbert space $\mathcal{H}=L^2(X,\mu)\otimes\ell^2(\Gamma)$ by sending 
\begin{equation*}
a\mapsto L(a):=a\otimes 1\,\,,\,\, u_g\mapsto L(u_g):=\sigma_g\otimes \lambda_g \quad\text{and} \quad
\end{equation*}
\begin{equation*}
\text{and} \quad L(u_g)L(a)L(u_g)^*=L(\sigma_g(a)).
\end{equation*}
The \textbf{group measure space von Neumann algebra $L^{\infty}(X,\mu)\rtimes \Gamma$ associated with} $\Gamma\curvearrowright(X,\mu)$ or \textbf{crossed product}  is the von Neumann algebra generated by 
\[
L(L^{\infty})\cup \{L(u_g)\,|\,g\in\Gamma\}.
\]
In particular, the elements in $L^{\infty}(X,\mu)\rtimes \Gamma$ may be identified to elements of $L^2(X,\mu)\otimes \ell^2(\Gamma)$ by $x\mapsto xU_e$ and hence are written as
\[
x= \sum_{g\in\Gamma} x_gu_g,
\]
with $\sum_{g\in\Gamma} \|x_g\|^2_{L^2(X,\mu)}<\infty$.  The coefficient $x_g\in L^{\infty}(X,\mu)$ are called \textbf{Fourier coefficients} of $x$ and the $u_g$ are called the \textbf{canonical unitaries of the crossed product}.
With the trace defined by 
\[
\tau(x)=\langle xu_e, u_e\rangle= \int_X x_e d\mu \quad \text{where} \quad x=\sum_{g\in\Gamma}x_gu_g.
\]
In particular, the group von Neumann algebra is a specific case when $X$ is just a singleton.

\section{Tensor product}
\begin{definition}
Let $M_1\in\mathcal{B}(H_1)$ and $M_2\in\mathcal{B}(H_2)$ be von Neumann algebras. The algebraic tensor product $M_1\odot M_2$ of $M_1$ and $M_2$ is defined by
\begin{equation*}
(x_1\otimes x_2)(\xi_1\otimes \xi_2)=(x_1\xi_1\otimes x_2\xi_2)
\end{equation*}
for any $x_i\in M_i$, $\xi_i\in H_i$ and $i=1,2$.
Obviously, $M_1\odot M_2$ is a $*$-algebra and its SOT-closure gives a von Neumann algebra acting on $H_1\odot H_2$. We call it \textbf{von Neumann tensor product} denoted by
\[
M_1\bar{\otimes} M_2.
\] 
\end{definition}

There is a celebrating theorem established by Tomita in 1960.

\begin{thm}[Tomita's Commutant theorem] Let $H_1, H_2$ be Hilbert spaces.
Let $M_1\subset \mathcal{B}(H_1)$ and $M_2\subset \mathcal{B}(H_2)$ be von Neumann algebras. Then
\begin{equation*}
(M_1\bar{\otimes}M_2)'=M_1'\bar{\otimes}M_2'
\end{equation*}
\end{thm}

According to Tomita's Theorem, we have that $M_1\bar\otimes M_2$ is a factor if each component $M_i$ needs to be a factor for $i=1, 2$. We, furthermore, have the following basic proposition.
\begin{prop} Given any von Neumann algebras $M_1$ and $M_2$. Then
\begin{enumerate}
\item If $M_1$ and $M_2$ are tracial factors, then so is $M_1\bar\otimes M_2$;
\item If $M_1$ and $M_2$ are $II_1$ factors, then so is $M_1\bar\otimes M_2$.
\end{enumerate}
\end{prop}

\begin{definition}
Let $M$ be a II$_1$ factor. We says $M$ is \textbf{prime} provided that if $M$ is isomorphic to a tensor product $M_1\bar\otimes M_2$ of von Neumann algebras $M_1, M_2$ then either $M_1$ or $M_2$ is finite dimensional.
\end{definition}

In the same spirit with Choda's Galois correspondence theorem \cite{Ch78}. Ge obtained a splitting theorem for tensors that we recall below. This is instrumental in deriving some of main results in this thesis.

\begin{thm}[Theorem A in \cite{Ge96}]\label{Ge96theoremA}
If $M$ is a finite factor, $N$ is a finite von Neumann algebra, and $B$ is a von Neumann sub algebra of $M\bar\otimes N$, there exists a von Neumann sub algebra $N_0$ of $N$ such that 
\begin{equation*}
B=M\bar\otimes N_0
\end{equation*}
\end{thm}

\section{Conditional expectation}

\begin{thm} [GNS-Construction]
Let $A^*$ be a C$^*$-algebra and $\varphi$ a positive linear functional on $A$. Then there exists a Hilbert space $L^2(A, \varphi)$ and a unique (up to equivalence) representation 
\begin{equation*}
\pi:A\rightarrow \mathcal{B}(L^2(A,\varphi))
\end{equation*}
with a unit cyclic vector $1_{\varphi}\in L^2(A,\varphi)$ such that
\begin{equation*}
\varphi(x)=\langle \pi(x)1_{\varphi}, 1_{\varphi}\rangle \quad\text{for all} \quad x\in A.
\end{equation*}
\end{thm}

Throughout the section $N$ denote a finite von Neumann algebra withe a fixed faithful normal trace $\tau$ and $B$ denote a von Neumann subalgebra of $N$. Using GNS construction, we can define the Hilbert space $L^2(N)$ which is defined over the dense linear subspace $N$ by 
\begin{equation*}
\langle x,y\rangle=\tau(xy^*) \quad\text{for all} \quad x,y\in N.
\end{equation*}
This $L^2(B)$ is a Hilbert subspace of $L^2(N)$ with the restricted inner product on $L^2(N)$.
Denote by $e_B:L^2(N)\rightarrow L^2(B)$ be the canonical orthogonal projection. We define
\[
E_B=e_B|_{N}.
\]
For the further use, we recall the following basic properties of this projection

\begin{thm}\label{conditonalexpectation}\
Let $B\subset N$ be von Neumann subalgebras.
\begin{enumerate}
\item $E_B=e_B|_N$ is a norm reducing map from $N$ onto $B$ with $E_B(1)=1$;
\item $E_B(bxc)=bE_B(x)c$ for all $x\in N$ and $b,c\in B$;
\item $\tau(xE_B(y))=\tau(E_B(x)E_B(y))=\tau(E_B(x)y)$ for all $x\in N$;
\item $\{e_B\}'\cap N=B$ and $B'=(N'\cup \{e_B\})''$;
\item $E_B$ is normal complete positive map;
\item $e_BJ=Je_B$ and $E_BJ=JE_B$
\item For the uniqueness, if $\phi:N\rightarrow B$ with 
\begin{equation*}
\phi(b_1xb_2)=b_1\phi(x)b_2 \quad \text{and}\quad \tau(\phi(x))=\tau(x)
\end{equation*}
 for all $x\in N$ and $b_1, b_2\in B$, then $\phi=E_B$.
\end{enumerate}
\end{thm}

\begin{definition}
Let $B\subset N$ be finite von Neumann algebras. From Theorem \ref{conditonalexpectation}, the \textbf{conditional expectation} $E_B:N\rightarrow B$ is defined by $E_B=e_B|_N$.
\end{definition}

Below we record some conditional expectation that will be useful subsequently.

(1) Let $\Lambda<\Gamma$ be groups. Consider $L(\Lambda)\subset L(\Gamma)$.We have
$E_{L(\Lambda)}(x)=\sum_{g\in\Lambda}x_gu_g$ where $x=\sum_{g\in\Gamma}x_gu_g\in L(\Gamma)$.

(2) Let $L^{\infty}(X)\rtimes \Gamma$ be a crossed product.
The conditional expectation $E_{L(\Gamma)}:L^{\infty}(X)\rtimes \Gamma\rightarrow L(\Gamma)$ is defined by
$E_{L(\Gamma)}(\sum x_gu_g)=\sum_{g\in \Gamma} \tau(x_g)u_g$ where $x=\sum_{g\in\Gamma}x_gu_g\in L^{\infty}(X)\rtimes \Gamma$, $E_{L^{\infty}(X)}(\sum x_gu_g)=x_e$.

(3) Let $N\subset M$ be finite von Neumann algebras and $p\in N$ be a projection. Define $E_{pNp}:pMp\rightarrow pNp$ by $E_{pNp}(x)=pE_{N}(x)E_{N}(p)^{-1}p$ for all $x\in pMp.$ 

Note that $\|x\|^2_{2,p}=\tau(p)^{-1}\|x\|_2^2$
 where $\|\cdot\|_2, \|\cdot\|_{2,p}$ are the norms on $L^2(M)$ and $L^2(pMp)$ respectively.

(4) Let $N\subset M$ be finite von Neumann algebras and $p$ be a projection in $N'\cap M$. Define $E_{Np}:pMp\rightarrow Np$ by
 $E_{Np}(x)=E_{N}(x)E_{N}(p)^{-1}p.$ for all $x\in pMp.$
 
 In the cases (3) and (4) if $N$ is a factor, then $E_{N}(p)=\tau(p)1.$

To study structural property  of inclusions of von Neumann algebras, an important tool is the associated basic construction. This algebra was introduced by  E. Christensen in order to study perturbations of algebras and later was used to great extended theory of finite index subfactors by V.F.R. Jones. The basic construction plays a key roles in Popa's deformation/rigidity  theory especially in the intertwining technique that we will see use in this dissertation.

\begin{definition}
If $B$ is a von Neumann subalgebra of a finite von Neumann algebra $N$ with faithful normal trace $\tau,$ the \textbf{basic construction} from the inclusion $B\subset N$ is defined to be the von Neumann algebra $\langle N,e_B\rangle:=(N\cup \{e_B\})''$.
\end{definition}

\begin{thm}
Let $B$ be a von Neumann subalgebra of finite von Neumann algebra $N$ with a fixed faithful normal trace $\tau.$ Then $\langle N,e_B\rangle$ is a semifinite von Neumann algebra with a faithful semifinite normal trace $Tr$ satisfying the following properties:
\begin{enumerate}
\item $\langle N,e_B\rangle=JB'J$, $\langle N,e_B\rangle'=JBJ$, and the $*$-subalgebra $Ne_BN=\text{span}\{xe_By\,|\, x,y\in N\}$ is weakly dense in $\langle N,e_B\rangle$;
\item the central support of $e_B$ in $\langle N,e_B\rangle$ is $1$;
\item $e_B\langle N,e_B\rangle e_B=Be_B$;
\item $e_BN$ and $Ne_B$ are weakly and strongly dense in respectively $e_b\langle N,e_B\rangle$ and $\langle N,e_B\rangle e_B$;
\item the map $x\mapsto xe_B:N\rightarrow Ne_B\subset \langle N,e_B\rangle e_B$ is injective;
\item $Tr(xe_By)=\tau(xy)$ for all $x,y\in N$;
\item $Ne_BN$ is dense in $L^2(\langle N,e_B\rangle, Tr)$ in $\|\cdot\|_{2,Tr}$-norm.
\end{enumerate}
\end{thm}

\section{Amplification}
Let $M\subset \mathcal{B}(H)$ be a von Neumann algebra 
For every $n\geq 1$, let $M_n(M)$ be a space of $n\times n$ matrices with entries in $M$. Clearly, $M_n(M)\subset \mathcal{B}(H^{\oplus n})$. Moreover, it is a straightforward proof to show that $M_n(M)$ is also a von Neumann algebra.
If $M$ is a type II$_1$ factor then $M_n(M)$ is also a type II$_1$ factor.

Denote $\operatorname{Tr}_n\otimes \tau$ its trace defined by
\begin{equation*}
(\operatorname{Tr}_n\otimes \tau)([x_{ij}])=\displaystyle\sum_{i}\tau(x_{ii}).
\end{equation*}
Moreover, we embed $M_n(M)$ into $M_{n+1}(M)$ by putting the zero entries in the last row and the last column, we obtain the increasing algebras

\begin{equation*}
\mathcal{M}(M) = \displaystyle\bigcup_{n\geq 1}M_{n}(M)
\end{equation*}

For any two projections $p,q\in\mathcal{M}(M),$ there is an $n\geq 0$ such that both $p$ and $q$ must belong to  $M_n(M).$ Since $M_{n}(M)$ is a factor and the trichotomy property for projections on factors, we have
\begin{center}
$p$  and $q$ are equivalent \quad if and only if \quad $(\operatorname{Tr}_n\otimes \tau)(p)=(\operatorname{Tr}_n\otimes \tau)(q)$
\end{center} 
This follows that
\begin{equation*}
p\,\mathcal{M}(M)\, p= p\,M_n(M)\, p\simeq q\,M_n(M)\,q.
\end{equation*}
Define
\begin{equation*}
M^t=p\,M_n(M)\,p, \quad\text{where}\quad t=(\operatorname{Tr}\otimes \tau)(p).
\end{equation*}
It is not too hard to check that $M^t$ is well-defined for every $t>0$ and unique up to isomorphism. We call $M^t$ an \textbf{amplification} of $M$ by $t$.
\begin{thm}
Let $M, M_0$ be II$_1$ factors and $s,t>0$. Then the following hold:
\begin{enumerate}
\item[(a)] $(M\bar\otimes M_0)^t=M\bar\otimes M_0^t=M^t\bar\otimes M_0$.
\item[(b)] $(M^s)^t=M^{st}$.
\item[(c)] $M\bar\otimes M_0=M^t\otimes M_0^{1/t}$.
\end{enumerate} 
\end{thm}

\begin{cor}
Given two groups $\Gamma_1, \Gamma_2$ and $t>0$, we have the relation
\begin{equation*}
L(\Gamma_1\times\Gamma_2)=L(\Gamma_1)\bar\otimes L(\Gamma_2)=L(\Gamma_1)^t\bar\otimes L(\Gamma_2)^{1/t}.
\end{equation*}
\end{cor}

\section{Ultrapower von Neumann algebras and property gamma}
In this section we introduce the ultrapower von Neumann algebra $N^{\omega}$ associated to a given von Neumann algebra $N$.
This is an important tool that provides algebraic framework to understand various asymptotic properties such as central sequence.
 We fix a \textit{free ultrafilter} $\omega$ on $\mathbb{N}$. Recall that $\omega$ is an element of $\beta\mathbb{N}\setminus \mathbb{N}$ where $\beta\mathbb{N}$ is the Stone-Cech compatification of $\mathbb{N}.$ For any bounded sequence $(c_n)$ of complex numbers, $\lim_{\omega} c_n$ is defined as the value at $\omega$ of this sequence, viewed as a continuous function on $\beta\mathbb{N}.$

Let $(M_n,\tau_n)$ is a sequence of tracial von Neumann algebras. The product algebra $\Pi_{n\geq 1} M_n$ is the C$^*$-algebra of bounded sequences $x=(x_n)_n$ with $x_n\in M_n$ for every $n$, endowed with the norm $\|x\|=\sup_n\|x_n\|.$ The (tracial) \textbf{ultraproduct} $\Pi_{\omega} M_n$ is the quotient of $\Pi_{n\geq} M_n$ by the ideal $I_{\omega}$ of all sequences $(x_n)_n$ such that $\lim_{\omega}\tau_n(x_n^*x_n)=0$. It is easily seen that $I_{\omega}$ is a normed closed two-sided ideal, so that $\Pi_{\omega}M_n$ is a C$^*$-algebra. If $x_{\omega}$ denotes the class of $x\in \Pi_{n\geq 1}M_n$, then $\tau_{\omega}(x):= \lim_{\omega}\tau_n(x_n)$ defines without ambiguity a faithful tracial state on $\Pi_{\omega} M_n$. We set $\|y\|_{2,\omega}=\tau_{\omega}(y^*y)^{1/2}$ whenever $y\in \Pi_{\omega} M_n$.

When $(M_n,\tau_n)=(M,\tau)$ for all $n$, we set $M^{\omega}=\Pi_{\omega}M$ and we says that $(M^{\omega},\tau_{\omega})$ is the (tracial) \textbf{ultrapower} of $(M,\tau)$ along $\omega$.

\begin{prop}
We have the followings.
\begin{enumerate}
\item $(\Pi_{\omega}M_n,\tau_{\omega})$ is a tracial von Neumann algebra.
\item If $M_n$ are finite factors such that $\lim_n \dim M_n=+\infty$, then $\Pi_{\omega} M_n$ is a II$_1$ factor.
\end{enumerate}
\end{prop}

Next we recall Murray-von Neumann property \textit{Gamma} associates with a von Neumann algebra.
This was the first invariant introduced to distinguish the hyperfintie II$_1$ factor $\mathcal{R}$ from the free group factor $L(\mathbb F_{2})$. This showed the existence of 
non-hyperfinite II$_1$ factors. 


\begin{definition}
A II$_1$ factor $M$ is said to have \textbf{Property Gamma} if given $\varepsilon>0$ and $x_1,\ldots, x_k\in N$, there exists a trace zero unitary $u\in M$ such that 
\begin{equation*}
\|ux_1-x_iu\|_2<\varepsilon, \quad 1\leq i\leq k.
\end{equation*}
An alternative formulation is the existence, for a fixed but arbitrary finite set $F \subset M$, of a sequence $\{u_n\}_{n=1}^{\infty}$ of trace zero unitaries in $N$ satisfying 
\begin{equation*}
\lim_{n\rightarrow \infty} \|u_nx-xu_n\|=0,\quad x\in F.
\end{equation*}
\end{definition}

\begin{thm}[\cite{Mc69}]
Let $M$ be a separable II$_1$ factor and let $\omega$ be free ultrafilter on $\mathbb{N}$. The following conditions are equivalent:
\begin{enumerate}
\item $M$ has Property Gamma;
\item $M'\cap M^{\omega}\neq \mathbb{C}1$;
\item $M'\cap M^{\omega}$ is diffuse.
\end{enumerate}
\end{thm}

\begin{definition}[\cite{Mc69}]
Let $M$ be a separable II$_1$ factor. For $\omega$ be free ultrafilter on $\mathbb{N}$, if the central sequence algebra $M'\cap M^{\omega}$ is non-abelian then $M\cong M\bar\otimes \mathcal{R}$ and $M$ is said to be \textbf{McDuff}.
\end{definition}

We finish this section by recording the important result for our development. 

\begin{thm}[Theorem 3.1in \cite{CSU13}]\label{CSU13theorem3.1}
Let $\Gamma$ ba a countable discrete group together with a family of subgroups $\mathcal{G}$ such that satisfies condition $\textbf{NC}(\mathcal{G})$. Let $(A,\tau)$ be any amenable von Neumann algebra equipped that $\omega$ is a free ultrafilter on the positive integers $\mathbb{N}$.

Then for any asymptotically central sequence $(x_n)_n\in M'\cap M^{\omega}$, there exists a finite subset $\mathcal{F}\subset \mathcal{G}$ such that $(x_n)_n\in \vee_{\Sigma\in\mathcal{F}}(A\rtimes \Sigma)^{\omega}\vee M$ (i.e. the von Neumann subalgebra of $M^{\omega}$ generated by $M$ and $(A\rtimes\Sigma)^{\omega}$ for $\Sigma\in \mathcal{F}$).
\end{thm}


\chapter{Intertwining Results in Amalgamated Free Product von Neumann Algebras} \label{chapter:chapterlabeltwo} 

\section{Popa's intertwining techniques} 

Over a decade, S. Popa has developed the following powerful method in \cite[Theorem 2.1 and Corollary 2.3]{Po03} to identify intertwines between arbitrary subalgebras of tracial von Neumann algebras. 

In order to study the structural theory of von Neumann algebras, S. Popa introduced the following concept of the intertwining subalgebras which has been very instrumental in the recent development in the classification of von Neumann algebra.

\begin{thm}[Popa's intertwining by bimodule technique]\label{popaintertwing}
Let $(M,\tau)$ be a finite von Neumann algebra. 
Suppose $P, Q$ be von Neumann subalgebras of $M$. Then the following are equivalent:
\begin{enumerate}
\item There exist projections $p\in P$, $p\in Q$, a nonzero partially isometry $v\in pPq$ and a $*$-homomorphism $\psi:pPp\rightarrow qQq$ such that 
\begin{equation*}
\psi(x)v=vx\quad \text{for all $x\in pPp$.}
\end{equation*}
and such that $v^*v\in \psi(pPp)'\cap qMq$ and $vv^*\in (pPp)'\cap pMp$.
\item For any group $\mathcal G\subset  \mathcal U(P)$ such that $\mathcal G''= P$, there is no sequence $(u_n)_n\subset \mathcal G$ satisfying for all $x,y\in  M$
\[
\|E_{ Q}(xu_ny)\|_2\rightarrow 0.
\]
\item There exists a $Q$-$P$-submodule $\mathcal{H}$ of $L^2(M)$ with $\dim_Q\mathcal{H}<\infty$.
\item There exists a positive element $a\in \langle M, e_Q \rangle$; the basic construction with $\operatorname{Tr}(a)<\infty$ such that the ultraweakly closed convex hull of $\{w^*aw\,|\, w\in P\,\,\text{unitary}\}$ does not contain $0$.
\end{enumerate}
\end{thm}

If one of the conditions in Theorem \ref{popaintertwing} above holds, we say \textbf{$Q$ embeds in $P$ inside $M$} and denoted by $P\prec_M Q$. Otherwise, we write $P\not\prec_M Q$. In the condition (1) the partial isometry $v$ is also called an \textbf{intertwiner} between $P$ and $Q$. 

Moreover, if we have $Pp'\prec_M Q$ for any nonzero projection $p'\in P'\cap 1_PM1_P$, then we write $P\prec_M^s Q$.

Next we record several well-known important results that will be used in the subsequent sections.

 
\begin{thm}[Corollary F.14 in \cite{BO08}]\label{BO08appendix}
Let $M$  be a finite von Neumann algebra with separable predual. Suppose $(A_n)\subset M$ is a sequence of von Neumann subalgebras and $N\subset pMp$ be a von Neumann subalgebra such that $N\not\prec_M A_n$ for any $n$. Then there exists a diffuse abelian von Neumann subalgebra $B\subset N$ such that $N\not \prec_M A_n$ for any $n$.
\end{thm}

\begin{prop}\label{intertamalgam} Let $M=M_1\ast_P M_2$ be an amalgamated free product von Neumann algebra.
If for each $i$ there is a unitary $u_i\in \mathcal U(M_i)$ such that $E_P(u_i)=0$ then
\begin{center}
$M\nprec_M M_k$ \,\,\,for all\,\,\, $k=1,2$. 
\end{center}
\end{prop}
\begin{proof} Let  $u=u_1u_2\in \mathcal U(M)$. Using freeness and basic approximation properties one can see that $\lim_{n\rightarrow \infty}\|E_{M_k} (x u^n y)\|_2=0$ for all $x,y\in M$. Then Theorem \ref{popaintertwing} (b) gives the conclusion. 
\end{proof}

\begin{thm}[Lemma 2.2 in \cite{CI17}]\label{CI17lemma2.2}
Let $\Gamma_1, \Gamma_2\leq \Gamma$ be countable groups such that 
\[
L(\Gamma_1)\prec_{L(\Gamma)}L(\Gamma_2).
\]
Then there exists  $g\in\Gamma$ such that $[\Gamma_1:\Gamma_1\cap g\Gamma_2g^{-1}]<\infty$.
\end{thm}


\chapter{Finite index inclusions of von Neuman algebras} \label{chapter:chapterlabel3}

In this section we recall several basic facts from the pioneering work of V.F.R Jones \cite{Jo81} on the theory of finite index inclusion of factors.

\begin{definition}
Let $B\subset M$ be an inclusion of finite von Neumann algebras. The a set $(m_i)_{1\leq i\leq n}\in M$ is called a (left) \textbf{Pimsner-Popa} basis if $m\in M$ has a unique expression form
\[
m=\sum_{i=1}^n m_ib_i
\]
where $b_i\in p_iB$. 
\end{definition}

\begin{thm}
Let $M$ be a II$_1$ factor and $B\subset M$ a von Neumann subalgebra.  Then $L^2(M)_B$ is finite generated  if and only if $m_1,\ldots,m_n\in M$ such that
\begin{enumerate}
\item[(i)] $E_{B}(m_i^*m_j)=\delta_{i,j}p_j$ is a projection in $B$ for all $i,j$;
\item[(ii)] $\sum_{1\leq i\leq n} m_ie_Bm_i^*=1$.
\end{enumerate}
If these conditions hold, we have $\sum_{1\leq i\leq n}m_im_i*=\dim(L^2(M)_B)1$ and $x=\sum_{1\leq i\leq n}m_iE_B(m_i^*x)=1$ for every $x\in M$.
\end{thm}

\begin{definition}
Let $B$ be a subfactor of a II$_1$ factor $M$. The \textbf{Jones' index} of $B$ in $M$ is defined as the dimension of $L^2(M)$ as a left $B$-module, i.e.,
\begin{equation*}
[M:B]=\dim_{\mathbb{C}}(L^2(M)_B).
\end{equation*}

\end{definition}
By the definition, we have
$[M:B]$ is finite \textit{if and only if} $\langle M,e_B\rangle$ is a type II$_1$ factor \textit{if and only if}  $L^2(M)_B$ is finitely generated.

\begin{thm}[Downward basic construction, Lemma 3.1.8 in \cite{Jo81}]\label{downwardbasicconstruction}
Let $N\subset M$ be II$_1$ factors such that $[M:N]<\infty$. Then there exists a subfactor $P\subset N$ and a projection $e_P\in M$ such that 
\begin{itemize}
\item $E_P(e_P)=\tau(e_P)1$,
\item $e_pxe_P=E_p(x)e_P$ for all $x\in N$, and  
\item $M=\langle N,e_P\rangle$.
\end{itemize}
\end{thm}

While V.F.R. Jones defined the notation of finite index on factors, Pimsner - Popa found a more probabilistic general notion of finite index that works for all inclusions of finite von Neumann algebras.

\begin{definition}[\cite{PP86}]\label{theorem2.2PP86}
If $B\subset M$ is a subfactor of the type II$_1$ factor, then
\begin{align*}
[M:B]^{-1}
&= \inf\big\{ \|E_B(x)\|^2_2/\|x\|^2_2 \,|\, x\in M_+,  x\neq 0\big\}
\end{align*}
with the convention $\infty^{-1}=0.$
If $[M:B]\neq 0$ then we says that $B\subset M$ has textbf{finite index} or is an finite index inclusion. In the case that $B\subset M$ are II$_1$ factors then it coincides with the notion of indexes by Jones.
\end{definition}

For the following proposition, we record some basic properties of finite index inclusions of von Neumann algebras that will be needed throughout our work. Even if they are well known, we also include their proofs for the sake of completeness.

\begin{prop}\label{finiteindexbasicprop} 
Let $N\subset M$ be von Neumann algebras with $[M:N]<\infty$. Then the following hold:\begin{enumerate}
\item \label{20} If $N$ is a factor, then 
\[
\operatorname{dim}_{\mathbb C}(N'\cap M)\leq [M:N]+1.
\] 
\item \label{20'}\cite[1.1.2(iv)]{Po95} If $\mathcal Z(M)$ is purely atomic\footnote{The unit 1 can be expressed as a sum of minimal projection} then $\mathcal Z(N)$ is also purely atomic, .  
\item\label{20''} \cite[1.1.2(ii)]{Po95} If $N$ is a factor and $r\in N'\cap M$ then 
\[
[rMr:Nr]\leq \tau(r) [M:N]<\infty.
\]
\end{enumerate}
\end{prop}

\begin{proof} 

(\ref{20}) Fix $0\neq p\in N'\cap M$ a nonzero projection. Since $N$ is a factor then $E_N(p)=\tau(p)1$. As $[M:N]<\infty$, we have 
\[
\tau(p)^2=\|E_N(p)\|_2^2\geq [M:N]^{-1}\|p\|_2^2=[M:N]^{-1}\tau(p).
\]
Since $p$ is an arbitrary projection in $N'\cap M$, we obtain $\tau(p)\geq [M:N]^{-1}$ for all projections $p\in N'\cap M$. Hence, 
\[
\operatorname{dim}_{\mathbb C}(N'\cap M)\leq [M:N]+1.
\]  

(\ref{20'}) Let $p\in \mathcal Z(N)$ be a maximal projection such that $\mathcal Z(N)p$ is purely atomic and $\mathcal Z(N)(1-p)$ is diffuse. To prove the conclusion it suffices to show that $q=1-p$ vanishes. Since the inclustion
$N\subset M$ is finite index, we have $qNq\subset qMq$ is finite index. This implies that $qMq\prec_{qMq} qNq$. Hence,
\[
qNq'\cap qMq\prec_{qMq} qMq'\cap qMq=\mathcal Z(M)q.
\]
Therefore, $\mathcal Z(N)q\prec \mathcal Z(M)q$. Since $\mathcal Z(M)$ is purely atomic, it follows that there exists a minimal projection of $\mathcal Z(N)$ under $q$. This forces $q=0$, as desired.

(\ref{20''}) Since $r\in N'\cap M$ and $N$ is a factor, we have $E_N(r)=\tau(r)1$. Thus, 
\[
E_{Nr}(rxr)= \tau(r)^{-1} E_{N}(rxr)  r\quad\text{for all}\quad x\in M.
\]
Hence, we have 
\begin{align*}
\|E_{Nr}(rxr)\|^2_{2,r}
&=\tau(r)^{-1} \|E_{Nr}(rxr)\|_2^2\\
&=\tau(r)^{-1}\big(\tau(r)^{-1} \|E_N(rxr)r\|_2^2\big)\\
&=\tau(r)^{-2} \|E_N(rxr)r\|_2^2\\
&\geq \tau(r)^{-2} [M:N]^{-1}\|rxr\|^2_2\\
&=\tau(r)^{-1} [M:N]^{-1}\|rxr\|^2_{2,r}
\end{align*}
which shows $[rMr:Nr]\leq \tau(r)[M:N]$.  
\end{proof}

\begin{definition}\label{virtuallyprimedef}
Let $M$ be a factor. We say $M$ is \textbf{virtually prime} if $A,B\subset M$ are commuting diffuse subfactors of $M$, then $[M:A\vee B]=\infty$.
\end{definition}

\begin{lem}\label{finiteindeximage}
Let $N\subset M$ be a finite index inclusion of II$_1$ factors. Then one can find projections $p\in M$, $q\in N$, a partial isometry $v\in M$, and a unital injective $*$-homomorphism 
$\phi:pMp \rightarrow qNq$ such that
\begin{enumerate} 
\item $\phi(x)v=vx \text{ for all } x\in pMp$, and
\item $[qNq: \phi(pMp)]<\infty$.
\end{enumerate}
\end{lem}
\begin{proof} 
Since $ [M:N]<\infty$ then $M\prec_M N$. Thus there exist projections $p\in M$, $q\in N$, a partial isometry $v\in M$, and a unital injective $*$-homomorphism $\phi:pMp \rightarrow qNq$ so that 
\begin{equation}\label{11}
\phi(x)v=vx \quad\text{ for all } \quad x\in pMp.
\end{equation}
Denoting by $Q=\phi(pMp)\subset qNq$, we notice that $vv^*\in Q'\cap qMq$ and $v^*v=p$. Moreover by restricting $vv^*$ if necessary we can assume wlog the support projection of $E_N (vv^*)$ equals $q$. Also from the condition (\ref{11}), we have that
\begin{equation*}
Qvv^*=vMv^*=vv^*Mvv^*
\end{equation*}
Since $M$ is a factor, passing to relative commutants we have 
\begin{align*}
vv^*(Q'\cap qMq)vv^*
&=(Qvv^*)'\cap vv^*Mvv^*\\
&=\mathcal Z(vv^*Mvv^*)\\
&= \mathbb C vv^*.
\end{align*}
Since $Q'\cap qNq\subset Q'\cap qMq$, there is a projection $r\in Q'\cap qN q$ such that 
\begin{equation*}
r (Q'\cap q N q) r=  Q r'\cap r N r=\mathbb C r.
\end{equation*}
 Since $q= \operatorname{s}(E_N(vv^*))$ one can check that $rv\neq 0$. Thus replacing $Q$ by $Qr$, $\phi(\cdot)$  by $\phi(\cdot)r$, $q$ by $r$, and $v$ by the partial isometry from the polar decomposition of $rv$ then the intertwining relation (\ref{11}) still holds with the additional assumption that $Q'\cap qM q=\mathbb Cq$. In particular, $E_{q N q}(vv^*)=c q$ where $c$ is a positive scalar.  

To finish the proof we only need to argue that $[qNq:Q]<\infty$. Consider the von Neumann algebra $\langle qNq, vv^*\rangle$ generated by $qNq$ and $vv^*$ inside $qMq$. Therefore we have the following inclusions  
\begin{equation*}
Q\subset qNq\subset \langle qNq, vv^*\rangle\subset qMq.
\end{equation*}
 Since $vv^* Mvv^* =Qvv^*$ then 
 \begin{equation*}
 vv^*  qNq \quad \text{and}\quad vv^*= Qvv^* .
 \end{equation*}
 Moreover, since $vv^*\in Q'\cap qMq$ and $E_{q N q}(vv^*)=c1$, one can check that $\langle qNq, vv^*\rangle$ is isomorphic to the basic construction of $Q\subset qNq$. Therefore, $Q\subset qNq$ has index $c$, hence finite.
 \end{proof}

\begin{lem}[Lemma 3.9 in \cite{Va08}]\label{lemma3.9Va08}
Let $(M,\tau)$ be a tracial von Neumann algebra and $A, B, N$ von Neumann subalgebras. Let $A\subset N$ be a finite index inclusion. Then the followings hold
\begin{enumerate}
\item [(1)] If $A\prec_M B$, then $N\prec_M B$.
\item [(2)] If $B\prec_M A$, then $B\prec_M A$.
\end{enumerate}
\end{lem}



\chapter{Amenability and Relative Amenability} \label{chapter:chapterlabel4}

Amenability is one of the important standard term in studying von Neumann algebra which was first introduced by Connes in 1976. In this chapter, we discuss about the amenabilities on groups and on von Neumann algebra. Finally we 
provide the concept of relative amenability for von Neumann algebras.

\section{Amenable groups with their von Neumann algebras}

\begin{definition} A group $\Gamma$ is said to be \textbf{amenable} if one of the following conditions holds:
\begin{enumerate}
\item[a)] there exists a left $\Gamma$- invariant mean on $\ell^{\infty}(\Gamma)$
\item[b)] there exists a sequence of unit vectors $(\xi_n)$ in $\ell^2(\Gamma)$ such that for every $g\in \Gamma$,
\begin{equation*}
\lim_n\|\lambda_G(g)\xi_n-\xi_n\|_2=0
\end{equation*}
\item[c)] there exists a sequence of finitely supported positive definite functions on $\Gamma$ which converges pointwise to $1$
\item[d)] \textit{F\o lner}; For any finite subset $E\subset \Gamma$ and $\varepsilon>0$ there is a finite subset $F\subset \Gamma$ such that 
\begin{equation*}
\max_{s\in E}\frac{|sF\,\Delta\, F|}{|F|}<\varepsilon.
\end{equation*}
\end{enumerate} 
\end{definition}


\section{Amenable von Neumann algebras}

\begin{definition}
A von Neumann algebra $M$ is said to be \textbf{amenable} or \textbf{injective} 
\begin{enumerate}
\item[a)] if it has a concrete representation as  a von Neumann subalgebra of some $\mathcal{B}(H)$ such that there exists c conditional expectation $\operatorname{E}:\mathcal{B}(H)\rightarrow M$.
\item[b)] for every inclustion $A\subset B$ of unital $C^*$-algebra, every unital completely positive  map $\phi:A\rightarrow M$ extends to a completely positive map from $B$ to $M$.
\item[c)] for any $\mathcal{B}(H)$ which contains $M$ as a von Neumann subalgebra, there is a conditional expectional expectation from $\mathcal{B}(H)$ onto $M$.
\end{enumerate}
\end{definition}

We use the word "amenable" to emphasis the analogy of the amenability for groups. By previous section we can show that  a countable group $\Gamma$ is amenable if and only if group von Neumann algebra $L(\Gamma)$ is amenable. 

\begin{thm}[\cite{Co76}]
The hyperfinite factor $R$ is amenable. 
\end{thm}

\section{Relative amenability for von Neumann algebras}

In practice we will use the following characterization, which comes from \cite{OP07} wihch was introduced by Ozawa-Popa.

\begin{definition}
Let $P\subset M$ be an inclusion of a von Neumann algebras. A state $\psi:M\rightarrow \mathbb{C}$ is \textbf{$P$-central}  if 
\[
\psi(mx)=\psi(xm)
\]
for every $x\in P$ and every $m\in M$.
\end{definition}

Following Section 2.2 in \cite{OP07}, we have the following definition

\begin{definition}
Let $P,Q$ be von Neumann subagebras of a tracial von Neumann algebra $(M,\tau)$ Then $P$ is \textbf{amenable relative to $Q$ inside $M$} and denoted by $P\lessdot_M Q$ if one of the following conditions holds:
\begin{enumerate}
\item[a)] there exists a conditional expectation from $\langle M,e_Q \rangle$ onto $P$ whose restriction to $M$ is $E_P^M$
\item[b)] there is a $P$-central state $\psi$ on $\langle M,e_Q \rangle$ such that $\psi{|_M}=\tau$
\item[c)] there is a $P$-central state $\psi$ on $\langle M,e_Q \rangle$ such that $\psi$ is normal on $M$ and faithful on $\mathcal{Z}(P'\cap M)$
\item[d)] there is a net $(\xi_i)$ of norm-one vector in $L^2(\langle M,e_Q \rangle)$ such that 
\begin{equation*}
\lim_i \|x\xi_i-\xi x\|=0\quad \text{for every}\quad x\in P
\end{equation*}
 and
\begin{equation*}
\lim_i\langle \xi_i x\xi_i\rangle=\tau(x) \quad\text{for every}\quad x\in M.
\end{equation*}
\item[e)] ${}_ML^2(M)_P$ is weakly contained in ${}_ML^2(M)\otimes L^2(M)_P$.
\end{enumerate}
\end{definition}

Moreover, if $M$ is amenable relative to $Q$ inside $M$, one simply says that \textbf{$M$ is amenable relative to $Q$} or that $Q$ is \textbf{co-amenable in $M$}. In particular, $M$ is amenable if and only if $M$ is amenable relative to $\mathbb{C}1.$

\begin{prop}[Ioana]
Let $P,Q$ be von Neumann subalgebras of a finite von Neumann algebra $(M, \tau)$. 
If  $P\prec^s_M Q$, then $P\lessdot_M Q$.
\end{prop}

\begin{prop}[Transitive property , Proposition 2.4 (3) in \cite{OP07}]
Let $P, Q, N\subset M$ be finite von Neumann algebras. If $N\lessdot_M P$ and $P\lessdot_M  Q$, then $N\lessdot_M  Q$.
\end{prop}

Next we record several important results that will be used in our subsequent development.

\begin{thm}[Theorem A in \cite{Va13}]\label{Va13theoremA}
Let $M=M_1*_{P}M_2$ be the amalgamated free product of the tracial von Neumann algebra $(M_i,\tau)$ with the common von Neumann subalgebra $P\subset M_i$ with respect to the unique trace preserving conditional expectations. Let $p$ be a nonzero projection, $A\subset pMp$ a von Neumann subalgebra that is amenable relative to one of the $M_i$ inside $M$. Then at least one of the following statement holds.
\begin{itemize}
\item $A\prec_M P$.
\item There is an $i\in\{1,2\}$ such that $\mathcal{N}_{pMp}(A)''\prec_M M_i$
\item $\mathcal{N}_{pMp}(A)''$ is amenable relative to $P$ inside $M$.
\end{itemize} 
\end{thm}

\begin{prop}[Proposition 2.7 in \cite{PV11}]\label{proposition2.7PV11}
Let $(M,\tau)$ be a tracial von Neumann algebra with von Neumann subalgebras $Q_1, Q_2\subset M$. Assume that $Q_1$ and $Q_2$ form a {commuting square} and that $Q_1$ is regular in $M$. If a von Neumann algebra $P\subset pMp$ is amenable relative to both $Q_1$ and $Q_2$, then $P$ is amenable relative to $Q_1\cap Q_2$.  
\end{prop}

\begin{lem}[Lemma 2.6 in \cite{DHI16}]\label{lemma2.6DHI16}
Let $(M,\tau)$ be a tracial von Neumann algebra, and $P\subset pMp$, $Q\subset M$ be von Neumann subalgebras. 
\begin{enumerate}
\item [(1)] Assume that $P$ is amenable relative tot $Q$. Then $Pp'$ is amenable relative to $Q$ for every projection $p'\in P'\cap pMp$.
\item [(2)] Assume that $p_0Pp_0p'$ is amenable relative to $Q$ for some projection $p_0\in P,$ $p'\in P'\cap pMp$. Let $z$ be the smallest projection belonging to $\mathcal{N}_{pMp}(P)'\cap pMp$ such that $p_0p'\leq z.$ Then $Pz$ is amenable relative to $Q$.
\item [(3)] Assume that $P\prec^s_M Q$. Then $P$ is amenable relative to $Q$.
\end{enumerate}
\end{lem}

\begin{lem}[Lemma 2.6 in \cite{IS19}]\label{lemma2.6IS19}
Let $(M,\tau)$ be a tracial von Neumann algebra and $Q\subset M$ a von Neumann subalgebra. Assume that there exists nets of von Neumann algebras $Q_n, M_n\subset M$ such that
\begin{enumerate}
\item[(1)] $Q\subset M_n\cap Q_n$ and ${}_{Q_n}L^{2}(M)_{M_n}\subset_{weak}{}_{Q_n}L^{2}(Q_n)\otimes_QL^2(M_n)_{M_n}$ for every $n$,
\item[(2)] $\lim_n \|x-E_{M_n}(x)\|_2= 0$ for every $x\in M$.
\end{enumerate}
If $P\subset M$ is a von Neumann subalgebra which is amenable relative to $Q_n$ inside $M$, for every $n$ then $P$ is amenable relative to $Q$ inside $M$.
\end{lem}


\chapter{Main Results} \label{chapter:chapterlabel6}


\section{Tensor product decompositions of amalgamated free products of von Neumann algebras}\label{AFPtensor}

In this section we preset a general result that completely describe all the tensor product decompositions for a large class of amalgamated free product von Neumann algebras $M_1\ast_P M_2$. Specifically, we will show that every tensor product product decomposition essentially splits the core $P$. This is a phenomenon that parallels results in Bass-Serre theory for groups. The precise statement is Theorem \ref{tensordecompafp}. However in order to prove our result we first need the following result which essentially relies on the usage of \cite[Theorem A]{Va13} (see also \cite[Theorem 7.1]{Io12}).  

\begin{thm}\label{intertwiningincore1} 
Let $M_1, M_2$ be tracial von Neumanna algebras with the common von Neumann subalgebra $P\subset M_i$ such that for each $i=1,2$ there is a unitary $u_i\in \mathcal U(M_i)$ so that $E_P(u_i)=0$. Let $M=M_1\ast_P M_2$ be the corresponding amalgamated free product von Neumann algebra and assume in addition that $M$ is not amenable relative to $P$ inside $M$.  Let $p\in M$ be a nonzero projection and assume $A_1,A_2 \subset pMp$ are two commuting diffuse subalgebras that $A_1\vee A_2\subset pMp$ has finite index. Then
\[
A_i\prec_M P\quad\text{for some}\quad i=1,2.
\]
\end{thm}

\begin{proof} Fix $A\subset A_1$ an arbitrary diffuse amenable subalgebra of $A_1$. Using Theorem \ref{Va13theoremA}, one of the following holds:
\begin{enumerate}
\item [(1)] \label{e1} $A\prec_M P$;
\item [(2)]\label{e2} $A_2\prec_M M_i$ for some $i=1,2$; or 
\item [(3)]\label{e3} $A_2$ is amenable relative to $P$ inside $M$.
\end{enumerate}
\noindent If (\ref{e2}) holds then either 
\begin{enumerate}[resume]
\item [(4)]\label{e4} $A_2\prec_M P$; or 
\item [(5)]\label{e5} $A_1\vee A_2\prec_M M_i$. 
\end{enumerate}
\noindent If (\ref{e5}) holds, since 
$[pMp:A_1\vee A_2]<\infty$, then we must have $M\prec_M M_i$. Then Proposition \ref{intertamalgam} will lead to a contradiction. If case (\ref{e3}) holds, then applying Theorem \ref{Va13theoremA} again 
we get one of the following
\begin{enumerate}[resume]
\item [(6)]\label{e6} $A_2\prec_M P$;
\item [(7)]\label{e7} $A_1\vee A_2$ is a amenable relative to $P$ inside $M$; or 
\item [(8)]\label{e8}$A_1\vee A_2\prec_M M_i$ for some $i.$
\end{enumerate}

\noindent If (\ref{e7}) holds, since $[pMp:A_1\vee A_2]<\infty$, it follows that $pMp$ is a amenable relative to $P$ inside $M$, contradicting the initial assumption. Notice that the condition (\ref{e8}) is similar to the condition (\ref{e5})  which was already eliminated before.
To summary, we have obtained that for any subalgebra $A\subset A_1$ amenable we have either
\begin{equation}
 A\prec_M P \quad\text{or} \quad A_2\prec_M P. 
 \end{equation} 
 
Here, suppose $A_1\not\prec_M P$
By using Theorem \ref{BO08appendix} and setting $A_n=P$ and $N=A_1$  We obtain that there exist a diffuse von Neumann subalgebra $B\subset A_1$ such that $B\not\prec_M P$. From above since $A$ is any arbitrary diffuse subalgebra, it is forced that $A_2\prec_M P$. So we can conclude that $ A_1\prec_M P$ or $A_2\prec_M P.$
\end{proof}

\begin{cor}\label{intertwiningcoregroupsafphnn} Let $\Gamma=\Gamma_1\ast_\Sigma \Gamma_2$ such that $[\Gamma_1:\Sigma]\geq 2$ and $[\Gamma_2:\Sigma]\geq 3$.
Denote by $M=L(\Gamma)$ let $p$ be a projection in $M$ and assume $A_1,A_2 \subset pMp$ are two commuting diffuse subalgebras such that $A_1\vee A_2\subset pMp$ has finite index. Then
\begin{equation}
A_i\prec_M L(\Sigma)\quad \text{for some} \quad i=1,2.
\end{equation}
\end{cor}

\begin{proof} Since $[\Gamma_1:\Sigma]\geq 2$ and $[\Gamma_2:\Sigma]\geq 3$ then by the proof of Theorem 7.1 in \cite{Io12} it follows that $L(\Gamma)$ is not amenable relative to $L(\Sigma)$. The conclusion follows then from Theorem \ref{intertwiningincore1}. 
\end{proof}

With these preparations at hand we are ready to prove the main theorem of this section.

\begin{thm}\label{tensordecompafp} 
Let $M=M_1\ast_P M_2$ be an amalgamated free product such that $M, M_1, M_2, P$ are II$_1$ factors and $[M_k:P]=\infty$ for all $k=1,2$. Assume $A_1,A_2 \subset M$ are diffuse factors such that 
$M=A_1\bar\otimes A_2.$
Then there exist tensor product decompositions 
\begin{center}
$P=C\bar \otimes P_0$,\quad $M_1 =C\bar\otimes M^0_1$,\quad and \quad $M_2 =C\bar\otimes M^0_2$
\end{center}
and hence $M= C\bar\otimes (M^0_1\ast_{P_0} M^0_2)$. Moreover, there exist $t>0$ and a permutation $\sigma \in \mathfrak S_2$ such that 
\begin{center}
$A_{\sigma(1)}^t\cong  C$ \quad and \quad $A_{\sigma(2)}^{1/t}\cong M^0_1\ast_{P_0} M^0_2$. 
\end{center}
\end{thm}

\begin{proof} By Theorem \ref{intertwiningincore1} we have that 
$A_i\prec_M P$ for some $i\in \{1,2\}$. Since $M=A_1\bar\otimes A_2$, by symmetry it suffices to assume $A_1\prec_M P$. It follows directly from the the definition that there exist nonzero projections $a\in A_1$, $p\in P$, a nonzero partial isometry $v\in M$, and a unital injective $\ast$-homomorphism 
\[
\Phi: aA_1 a\rightarrow pPp
\]
such that 
\begin{equation}\label{-1}
\Phi(x)v=vx \quad \text{for all}\quad x\in aA_1 a.
\end{equation}    

Shrinking $a$ if necessary we can assume there is an integer $m$ such that $\tau(p)=m^{-1}$. Letting $B=\phi(aA_1a)$, it is easy to chech that $vv^*\in B'\cap pMp$. Also we can assume wlog that  $\operatorname{s}(E_P(vv^*))=p$ and using factoriality of $A_i$ that  $v^*v= r_1\otimes r_2$. Thus by (\ref{-1}) there is a unitary $u\in M$ which is extended from $v$ so that 
\begin{equation}
\label{-3}Bvv^*=vA_1v^*= u (r_1 A_1 r_1 \otimes r_2) u^*.
\end{equation}
Passing to relative commutants we also have 
\begin{align}\label{-4}
vv^* (B'\cap pMp )vv^* 
&=vv^*B'vv^*\cap vv^*pMpvv^*\nonumber\\
&= (Bvv^*)'\cap vv^* Mvv^*\nonumber\\
&=(vA_1v^*)'\cap vMv^* \nonumber\\
& = u ((r_1 A_1 r_1 \otimes r_2 )'\cap (r_1 A_1 r_1 \otimes r_2 A_2 r_2 ))u^*\nonumber \\
&= u (r_1 \otimes r_2A_2 r_2)u^*.
\end{align}
Combing (\ref{-3}) and (\ref{-4}) together, we have 
\begin{align*}
vv^*(B\vee B'\cap pMp)vv^*
&= u (r_1 A_1r_1) \bar\otimes (r_2A_2 r_2 )u^*\\
&=vv^* Mvv^*.
\end{align*}
Letting $z$ be the central support of $vv^*$ in $B\vee B'\cap pMp$ we conclude that \begin{equation}\label{-2}(B\vee B'\cap pMp)z=zMz.
 \end{equation}
Note by construction we actually have $z\in \mathcal Z(B'\cap pMp)$. In addition, we have 
$p\geq z\geq vv^*$
and hence
\[
 p\geq  \operatorname{s}(E_P(z))\geq \operatorname{s}(E_P(vv^*))=p.
\]
 Thus  $\operatorname{s}(E_P(z))=p$. Also notice that $p\geq \operatorname{s}(E_{M_k}(z))\geq z\geq vv^*$.  For every $t>0$, denote $e^k_t=\chi_{[t,\infty )}(E_{M_k}(z))$.  Using relation (\ref{-2}) and \cite[Lemma 2.3]{CIK13} it follows that the inclusion $(B\vee B'\cap pM_k p)e^k_t\subset e^k_t M_k e^k_t$ is finite index. This, together with the assumptions and \cite[Lemma 3.7]{Va08} further imply that $(B\vee B'\cap pM_k p)e^k_t\nprec_{M_k} P$. But $e^k_t z$ commutes with $(B\vee B'\cap pM_k p)e^k_t$ and hence by \cite[Theorem 1.2.1]{IPP05} we have $e^k_tz\in M_k$. Since $e^k_t z\rightarrow z$ in $WOT$, as $t\rightarrow 0$, we obtain that $z\in pM_kp$, for all $k=1,2$. In conclusion $z\in pM_1p\cap pM_2p=pPp$ and hence $z=p$. Thus using factoriality and (\ref{-2}) we get that $pMp= B \bar \otimes (B'\cap pMp)$. Moreover,  we have $B\subset pPp\subset pMp= B \bar \otimes (B'\cap pMp)$ and since $B$ is a factor it follows from Theorem \ref{Ge96theoremA} that $pPp= B\bar\otimes (B'\cap pPp)$. 
Similarly one can show that $pM_kp= B\bar\otimes (B'\cap pM_kp)$ for all $k=1,2$. Thus, 
\begin{align*}
B'\cap pMp
&= (B'\cap pM_1p) \vee (B'\cap pM_2p)\\
&= (B'\cap pM_1p)\ast_{(B'\cap pPp)} ( B'\cap pM_2p).
\end{align*}
Combining these observations, we now have
\begin{equation*}\begin{split}
pMp
&= B\bar\otimes (B'\cap pMp)\\
& = B\bar \otimes ((B'\cap pM_1p)\ast_{(B'\cap pPp)} ( B'\cap pM_2p)) \\
&=(B\bar\otimes (B'\cap pM_1p))\ast_{B\bar \otimes (B'\cap pPp)} (B\bar \otimes ( B'\cap pM_2p)).
\end{split} \end{equation*} 
Tensoring by $M_m(\mathbb C)$ this further gives \begin{equation*}\begin{split}
M&= M_m(\mathbb C)\bar \otimes pMp\\
&= M_m(\mathbb C)\bar \otimes B\bar \otimes ((B'\cap pM_1p)\ast_{(B'\cap pPp)} ( B'\cap pM_2p)) \\&=(M_m(\mathbb C)\bar \otimes B\bar\otimes (B'\cap pM_1p))\ast_{M_m(\mathbb C)\bar \otimes B\bar \otimes (B'\cap pPp)} (M_m(\mathbb C)\bar \otimes B\bar \otimes ( B'\cap pM_2p))\end{split} 
\end{equation*} 
Letting 
\begin{center}
$C:= M_m(\mathbb C)\bar \otimes B$,\quad $P_0:= B'\cap pPp$, \quad and \quad $M^0_k:=B'\cap pM_kp$,
\end{center}
altogether, the previous relations show that 
\begin{center}
$P=C\bar \otimes P_0$, $M_1 =C\bar\otimes M^0_1$, $M_2 =C\bar\otimes M^0_2$, and $M= C\bar\otimes (M^0_1\ast_{P_0} M^0_2)$.
\end{center}
For the remaining part of the conclusion, notice that relations (\ref{-3}), (\ref{-4}) and $p=z(vv^*)$ show that 
\[
A_i^{\tau(r_1)} \cong B,  \quad A_{i+1}^{\tau(r_2)}\cong (B'\cap pMp)^{\tau(vv^*)}.
\]
Using amplifications these further imply that 
\[
A_i^{m \tau(r_1)} \cong C, \quad  A_{i+1}^{\tau(r_2)/ (m\tau(vv^*))}\cong M^0_1\ast_{P_0}M^0_2.
\]
Letting $t= m\tau(r_1)$ we get the desired conclusion.
\end{proof}


\section{Spatially commensurable von Neumann algebras}\label{commalg}

In the context of Popa's concept of weak intertwining of von Neumann algebras we introduce a notion of commensurable von Neumann algebras up to corners. This notion is essential to this work as it can be used very effectively to detect tensor product decompositions of II$_1$ factors (see Theorems \ref{fromrelcomtocomgroups} and \ref{virtualprod} below). It is also the correct notion which translate in the von Neumann algebraic language to the notion of commensurability for groups.

In the first part of section we build the necessary technical tools to prove these two results. Several of the arguments developed here are inspired by ideas from \cite{CdSS15} and \cite{DHI16}.  

\begin{definition} 
Let $P,Q\subset M$ (not necessarily unital) be inclusions of von Neumann algebras. We write $P\cong^{com}_M Q$ (and we say \emph{a corner of $P$ is spatially commensurable to a corner of $Q$}) if there exist nonzero projections $p\in P$, $q\in Q$, a nonzero partial isometry $v\in M$ and a $\ast$-homomorphism $\phi: pPp\rightarrow qQq $ such that 
\begin{eqnarray}
 \label{ccom1}& \phi(x)v=vx \quad \text{ for all }\quad x\in pPp\\
 \label{ccom2}&[qQq: \phi(pPp) ]<\infty\\
 \label{ccom3}& \operatorname{s}(E_Q(vv^*))=q.
\end{eqnarray} 

\noindent  When just the condition (\ref{ccom1}) is satisfied together with $\phi(pPp)=qQq$. In other words, $\phi$ is a $\ast$-isomorphism. We write $pPp\cong^{\phi,v}_M qQq$.
\end{definition}
\begin{remark} When $pPp$ is a II$_1$ factor then so is $\phi(pPp)$. By Proposition \ref{finiteindexbasicprop} (1), $\phi(pPp)'\cap qQq)$ is finite dimensional, so there exists $r\in \phi(pPp)'\cap qQq)$ such that $rv\neq 0$. Thus replacing $\phi(\cdot)$ by $\phi(\cdot)r$ and $v$ by the isometry in the polar decomposition of $rv$ one can check  (\ref{ccom1}) still holds. Also from Proposition \ref{finiteindexbasicprop} (3) it follows that $\phi(pPp)r\subset rQr$ is an finite index inclusion of II$_1$ factors. Hence throughout this article, whenever $P\cong^{com}_M Q$ and $P$ is a factor, we will always assume the algebras in (\ref{ccom2}) are II$_1$ factors.  
\end{remark}

For further use we recall the following result from \cite[Lemma 2.6]{CKP14}. 
\begin{prop}[Proposition 2.4]\cite{CKP14}\label{prop2.4CKP14}
Let $(M,\tau)$ be a tracial von Neumann algebra and let $z\in M$ be a nonzero projection. Suppose that $P\subset zMz$ and $N\subset M$ are von Neumann subalgebras such that $P\vee (P'\cap zMz)\subset zMz$ has finite index and that $P\prec_M N$. Then there exist a scalar $s>0$, nonzero projections $r\in N$, $p\in P$, a subalgebra $P_0\subset rNr$, and a $*$-isomorphism $\theta:pPp\rightarrow P_0$ such that the following properties are satisfied:
\begin{enumerate}
\item $P_0\vee (P'_0\cap rNr)\subset rNr$ has finite index;
\item there exist a nonzero partial isometry $v\in M$ such that 
\begin{center}
$r\operatorname{E}_N(vv^*)=\operatorname{E}_{N}(vv^*)r\geq sr$ \quad and \quad $\theta(pPp)v= P_0v=rvpPp$;
\end{center}
\item $\operatorname{E}_N(v(pP'p\cap pMp)v^*)''\subset P'_0\cap rNr$.
\end{enumerate}
\end{prop}
We record next a technical variation of \cite[Proposition 2.4]{CKP14} in the context of commensurable von Neumann algebras that will be essential to deriving the main results of this section.

\begin{lem}\label{intertwiningdichotomy} 
Let $\Sigma<\Gamma$ be groups where $\Gamma$ is icc. Assume $\mathcal Z(L(\Sigma))$ is purely atomic\footnote{The unit 1 can be expressed as a sum of minimal projection}, $r\in L(\Gamma)$ is a projection, and there exist commuting II$_1$ subfactors $P,Q\subset rL(\Gamma)r$  such that $P\vee Q\subset rL(\Gamma)r$ has finite index. If $P\prec_M L(\Sigma)$ then one of the following holds:
\begin{enumerate} 
\item\label{p1}  There exist projections $p\in P, e\in L(\Sigma) $, a partial isometry $w\in M$,  and a unital injective $*$-homomorphism $\Phi: pPp\rightarrow eL(\Sigma)e$ such that \begin{enumerate}
\item\label{11'''''} $\Phi(x)w=wx \text{ for all } x\in pPp$;
\item\label{11''''''} $\operatorname{s}(E_{L(\Sigma)}(ww^*))=e$;
\item \label{12'''''} If $B:=\Phi(pPp)$ then $B\vee (B'\cap eL(\Sigma)e) \subset eL(\Sigma)e$ is a finite index inclusion of II$_1$ factors.
 \end{enumerate}
\item\label{p2} $P\cong^{com}_{L(\Gamma)} L(\Sigma)$.
\end{enumerate}
\end{lem}

\begin{proof} 
From the assumption, $P\prec_{L(\Gamma)} L(\Sigma)$ so there exist projections $p\in P$, $q\in L(\Sigma)$, a nonzero partial isometry $v\in L(\Gamma)$,  and a unital injective $*$-homomorphism $\phi: pPp\rightarrow qL(\Sigma)q$ such that 
\begin{align}\label{11'''''a}
\phi(x)v=vx \quad\text{ for all } x\in pPp.
\end{align}
Let $C:=\phi(pPp)$. Note $v^*v\in pPp'\cap pL(\Gamma)p$, $vv^*\in C'\cap qL(\Gamma)q$ and we can also assume that 
\begin{equation*}
\operatorname{s}(E_{L(\Sigma)}(vv^*))=q
\end{equation*} 
Clearly $Q\subset P'$. Since $P\vee Q\subset rL(\Gamma)r$ has finite index, we have
$P\vee (P'\cap rL(\Gamma)r)$ also has finite index in $rL(\Gamma)r$. By Proposition \ref{prop2.4CKP14}, it implies that
\begin{align}\label{12'''''b}
C\vee (C'\cap qL(\Sigma)q) \subset qL(\Sigma)q\end{align} is also a finite index inclusion of algebras. 
By Proposition \ref{finiteindexbasicprop}(2),  $\mathcal Z(C'\cap qL(\Sigma)q)$ is purely atomic and there is a nonzero projection $e\in \mathcal Z(C'\cap qL(\Sigma)q)$ so that $ev\neq 0$ and we have either 
\begin{enumerate}[label=\roman*)]
\item
$ (C'\cap qL(\Sigma)q)e$ is a II$_1$ factor, or \label{13a}
\item$(C'\cap qL(\Sigma)q)e={\rm M}_n(\mathbb C) e$ for some $n\in \mathbb N$.\label{13b}\end{enumerate} 
Consider $\Phi: pPp\rightarrow Ce=:B$ given by
\begin{equation*}
\Phi(x)=\phi(x)e \quad\text{for all}\quad x\in pPp
\end{equation*}
 and let $w$ be the partial isometry in the polar decomposition of $ev$.  Then condition (\ref{11'''''a}) implies that 
 \begin{equation*}
 \Phi(x)w=wx\quad \text{for all}\quad x\in pPp.
\end{equation*} 

 Moreover, we have $evv^*e\leq ww^*$. Applying conditional expectation to the relation, with its properties we have
 \[
 E_{qL(\Sigma)q}(vv^*)e= E_{qL(\Sigma)q}(evv^*e)\leq E_{qL(\Sigma)q}(ww^*).
 \]
By considering support vectors, we obtain that 
\[
e=\operatorname{s}(E_{qL(\Sigma)q}(vv^*))e=\operatorname{s}(E_{qL(\Sigma)q}(vv^*)e)\leq \operatorname{s}(E_{qL(\Sigma)q}(ww^*)).
\]
Since $w$ is the partial isometry in the polar decomposition of $ev$, by its uniqueness, we have $eE_{qL(\Sigma)q}(ww^*)=E_{qL(\Sigma)q}(eww^*)=E_{qL(\Sigma)q}(ww^*)$. it follows that $\operatorname{s}(E_{qL(\Sigma)q}(ww^*))\leq e$ and therefore $\operatorname{s}(E_{L(\Sigma)}(ww^*))=e.$

Assume case \ref{13a} above. Using (\ref{12'''''b}), we have
\[
B\vee (B'\cap eL(\Sigma)e)= C e\vee (C'\cap qL(\Sigma)q)e\subset eL(\Sigma)e)\] 
is a finite index inclusion of II$_1$ factors. Altogether, these lead to possibility (\ref{p1}) in the statement.

Assume case \ref{13b} above.  Then relation (\ref{12'''''b}) implies that
\[
C=B e\subset  eL(\Sigma)e
\]
is a finite index inclusion which gives possibility (\ref{p2}) in the statement. 
\end{proof}


\begin{thm}[Claims 4.7-4.12 in \cite{CdSS15}]\label{fromrelcomtocomgroups} 
Let $\Sigma<\Lambda$ be finite-by-icc groups\footnote{A group $G$ is called \textit{finite-by-icc} if it has a normal subgroup $N$ that is finite and the quotient $G/N$ is icc.}. Also assume there exists $0\neq p\in \mathcal Z(L(\Sigma)'\cap L(\Lambda))$ such that $L(\Sigma)\vee (L(\Sigma)'\cap L(\Lambda))p\subset pL(\Lambda)p$ admits a finite Pimnser-Popa basis. Then there exists $\Omega<\Lambda$ such that 
\begin{center}
$[\Sigma,\Omega]=1$\quad and\quad$[\Lambda:\Sigma\Omega]<\infty$. 
\end{center} 
\end{thm}

The next result is a basic von Neumann's projections equivalence property for inclusions of von Neumann algebras. Its proof is standard and we include it only for reader's convenience. 

\begin{lem}\label{intinsubfactor} Let $N\subset (M, \tau)$ be finite von Neumann algebras, where $N$ is a II$_1$ factor. Then for every projection $0\neq e\in M$ there exists a projection $f\in N$ and a partial isometry $w\in M$ such that $e=w^*w$ and $ww^*=f$. 
\end{lem}

\begin{thm}\label{virtualprod} 
Let $\Sigma < \Gamma$ be countable groups, where $\Gamma$ is icc and $\Sigma$ is finite-by-icc. Let $r\in L(\Gamma)$ be a projection and let $P,Q\subset rL(\Gamma)r$ be commuting II$_1$ factors such that $P\vee Q\subset rL(\Gamma)r$ has finite index. If  $P\cong^{com}_{L(\Gamma)} L(\Sigma)$ 
then there exist a subgroup $\Omega< C_{\Gamma}(\Sigma)$ satisfying the following properties:
\begin{enumerate}
\item[(a)] \label{a1} $[\Gamma:\Sigma \Omega]<\infty$;
\item[(b)] \label{a3}$Q\cong^{com}_M L(\Omega)$.
\end{enumerate} 
\end{thm}

\begin{proof} Since $P\cong^{com}_{L(\Gamma)} L(\Sigma)$, by the definition there exist nonzero projections $p\in P,$ $q\in L(\Sigma)$, a nonzero partial isometry $v\in L(\Gamma)$, and an injective, unital $\ast$-homomorphism $\Phi: pPp\rightarrow eL(\Sigma)e$ so that \begin{enumerate}
\item[(1)]  $\Phi(x)v=vx$ for all $x\in pPp$, and 
\item[(2)]  $\Phi(pPp)\subset qL(\Sigma)q$ is a finite index inclusion of II$_1$ factors.
\end{enumerate}
Now we denote by $R:=\Phi(pPp)\subset qL(\Sigma)q$. Let $T \subset R\subset qL(\Sigma) q$ be the downward basic construction for inclusion $R\subset qL(\Sigma) q$. Since $[qL(\Sigma) q:R]<\infty$, according to Theorem \ref{downwardbasicconstruction} let $a\in T'\cap qL(\Sigma) q$ be the Jones' projection satisfying 
\begin{equation}
qL(\Sigma)q=\langle R,a \rangle\quad\text{and}\quad a L(\Sigma)a = Ta.\label{2eqs}
\end{equation}
 Also note that $[qL(\Sigma)q: R]=[R:T]$. As the $*$-homomorphism $\Phi: pPp\rightarrow qL(\Sigma)q$ is injective, the restriction 
$ \Phi^{-1}: T\rightarrow pPp$
is  an injective $\ast$-homomorphism such that $U:=\Phi^{-1}(T)\subset pPp$ is a finite Jones index subfactor and 
\begin{equation}\label{4}
\Phi^{-1} (x)v^*=v^*x\quad\text{for all}\quad x\in T.
\end{equation}  

Notice that $T\subset qL(\Sigma)q$ and the projection $a\in T'\cap qL(\Sigma)q$.
Let $\theta':Ta\rightarrow T$ be the $\ast$-isomorphism given by 
$\theta'(xa)=x$ for all $x\in T.$

We can check that $v^*a\neq 0$ and from the polar decomposition of $v^*a$, let $w_0$ be a nonzero partial isometry so that $v^*a=w^*_0|v^*a|$. Since from above we know $Ta=aL(\Sigma)a$, combining together with (\ref{4}) we have that the compostion map
\[
\theta=\Phi^{-1}\circ\theta': aL(\Sigma) a\rightarrow pPp
\]
is an injective $\ast$-homomorphism such that its image
\begin{equation*}
\theta(aL(\Sigma )a)=\Phi^{-1}\circ\theta'(aL(\Sigma)a)=\Phi^{-1}(T)=U\subset pPp \quad \text{and}
\end{equation*}  
\begin{equation}\label{3}
\theta (y)w^*_0=w^*_0y\quad\text{for all }\quad y\in aL(\Sigma) a.
\end{equation} 

By the assumption $P\vee Q \subset rL(\Gamma)r$ has finite index. It follows that $pPp\vee Qp \subset pL(\Gamma)p$ also has finite index as well. From (\ref{2eqs}) we have $U \subset pPp$ has finite index so it follows that $U\vee Qp\subset pL(\Gamma)p$ has finite index. Since these all are factors, it follows that
$U\vee Qp\subset pL(\Gamma)p$
admits a finite Pimsner-Popa basis. From construction we have
\[
U\vee Qp\,\,\subset\,\, U\vee (U'\cap pL(\Gamma)p)\,\,\subset\,\, pL(\Gamma)p
\]
and hence $U\vee Qp\subset U\vee (U'\cap pL(\Gamma)p)$ admits a finite Pimsner-Popa basis. Also since $U\vee Qp$ is a factor, we have by Proposition \ref{finiteindexbasicprop}(1) that 
\[
\operatorname{dim}_{\mathbb C}\bigg(\big[U\vee (U'\cap pL(\Gamma)p\big]\cap(U\vee Qp)'\bigg)<\infty.
\]
Since $[U\vee (U'\cap pL(\Gamma)p\big]\cap(U\vee Qp)'=[U'\cap pL(\Gamma)p]\cap(U\vee Qp)'$, we conclude that 
\[
\operatorname{dim}_{\mathbb C}\big([U'\cap pL(\Gamma)p]\cap(U\vee Qp)'\big)<\infty.
\]
Using Proposition \ref{finiteindexbasicprop}(3) for every minimal projection $b\in [U'\cap pL(\Gamma)p]\cap(U\vee Qp)'$, then we have
\[
(U\vee Qp) b\subset \big(U\vee (U'\cap pL(\Gamma)p)\big)b
\]
is a finite inclusion of II$_1$ factors.

\noindent\textit{Claim:} $Qb\subset (U'\cap pL(\Gamma)p)b$ has finite index.\\
Now we have known from above that $(U\vee Qp) b\subset \big(U\vee (U'\cap pL(\Gamma)p)\big)b$
is a finite inclusion.
Thus, by \ref{theorem2.2PP86} there exists $C_b>0$ such that for all $x\in U_+$ and $y\in (U'\cap pL(\Gamma)p)_+$ we have 
\begin{equation}\label{1201}
\|E_{U\vee Q b}(xyb)\|^2_{2,b}\geq C_b\|xyb\|^2_{2,b},
\end{equation}
where $\|\cdot\|_{2,b}$ is the norm on $L^2(bL^2(\Gamma)b)$.
Since $E_{U\vee Qp}(b)=\tau_p(b) p$ we have  
\begin{equation*}
E_{U\vee Q b}(zb)=E_{U\vee Qp}(zb) b \tau^{-1}_p(b)\quad\text{for all}\quad z\in  U\vee (U'\cap pL(\Gamma)p).
\end{equation*}
Thus for every $x\in U$ and $y\in (U'\cap pL(\Gamma)p)$ we have
\begin{align*}
E_{U\vee Qb}(xyb)
&=E_{U\vee Qp}(xyb) b \tau^{-1}_p(b)\\
&=xE_{U\vee Qp}(yb) b \tau^{-1}_p(b)\\
&=xE_{Qp}(yb) b \tau^{-1}_p(b)\\
&=xE_{Qb}(yb).
\end{align*}  
Also since $U$ is a factor, we can check that we have 
\[
\|xyb\|^2_2=\|x\|^2_2\|yb\|^2_2
\]
for all $x\in U$ and $y\in (U'\cap pL(\Gamma)p)$.
This further implies that 
\[
\|xyb\|^2_{2,b}=\|x\|^2_2\|yb\|^2_{2,b}.
\]
Using these formulas together with (\ref{1201}) we see that
\begin{align*}
\|x\|^2_2 \| E_{Q b}(yb)\|^2_{2,b}
&=\|x E_{Q b}(yb)\|^2_{2,b}\\
&=\|E_{U\vee Q b}(xyb)\|^2_{2,b}\\
&\geq C_b\|xyb\|^2_{2,b}\\
&=C_b\|x\|^2_2 \|yb\|^2_{2,b}
\end{align*}
and hence 
\[
\| E_{Q b}(yb)\|^2_{2,b}\geq C_b \|yb\|^2_{2,b}
\]
for all $y\in (U'\cap pL(\Gamma)p)_+$. Hence $Qb\subset U'\cap pL(\Gamma)p b$ is a finite index inclusion of II$_1$ factors for every minimal projection $b\in [U'\cap pL(\Gamma)p]\cap(U\vee Qp)'$. 
 
Choose a minimal projection $b\in [U'\cap pL(\Gamma)p]\cap(U\vee Qp)'$ so that $w^*=bw^*_0\neq 0$. Thus (\ref{3}) gives 
 \begin{equation}\label{3'}
\theta (y)w^*=w^*y\text{, for all }y\in aL(\Sigma) a.
\end{equation} 
Notice that $w^*w\in (U'\cap pL(\Gamma)p) b$  and $ww^*\in aL(\Sigma)a'\cap aL(\Gamma)a$. Let $u\in pL(\Gamma)p$ be a unitary so that $uw^*w=w$, then relation (\ref{3'}) entails  \begin{equation}\label{5}u U w^*wu^*=ww^* aL(\Sigma)a.
\end{equation} 
Passing through relative commutants we also have 
\begin{align}
uw^*w (U'\cap pL(\Gamma)p) w^*wu^*
&=ww^* (aL(\Sigma)a'\cap aL(\Gamma)a) ww^*\\ 
&= ww^* (L(\Sigma)'\cap L(\Gamma)) ww^*\label{6}
\end{align}
Altogether, (\ref{5}) and (\ref{6}) imply that 
\begin{equation}\label{7}\begin{split}uw^*w (U\vee (U'\cap pL(\Gamma)p)) w^*wu^*& =ww^* (a L(\Sigma)a \vee  (aL(\Sigma)a'\cap aL(\Gamma)a)) ww^*\\& = ww^* ( L(\Sigma) \vee  (L(\Sigma)'\cap L(\Gamma))) ww^*.\end{split}
\end{equation} 
Since from assumptions $pPp \vee Qp =p(P\vee Q)p\subset pL(\Gamma)p$ is a finite index and $U\subset pPp$ is finite index it follows that $U \vee Qp \subset pL(\Gamma)p$ has finite index as well. Also notice
\begin{align*}
U\vee Qp 
&\subset U\vee (P'\cap rL(\Gamma)r)p\\
&= U\vee (pPp'\cap pL(\Gamma)p)\\
&\subset U\vee (U'\cap pL(\Gamma)p)
\end{align*}
Thus $U\vee (U'\cap pL(\Gamma)p)\subset pL(\Gamma)p$ is finite index. Combining with (\ref{7}) we obtain 
\begin{equation*}
ww^*( L(\Sigma) \vee  (L(\Sigma)'\cap L(\Gamma))) ww^*\subset ww^* L(\Gamma)ww^*
\end{equation*}
is a finite index inclusion of II$_1$ factors. By using Theorem \ref{fromrelcomtocomgroups}, there exists a subgroup $\Omega<\Lambda$ such that
\begin{equation}
[\Sigma,\Omega]=1\quad \text{and} \quad [\Gamma:\Sigma\Omega]<\infty. \label{2eqs2}
\end{equation}
Since $\Gamma$ is an icc group, it follows that $\Sigma,\Omega$ also are icc groups as well; in particular, both $L(\Sigma)$ and $L(\Omega)$ are II$_1$ factors.
By Lemma \ref{intinsubfactor}, there exist unitaries $u_1\in U'\cap pL(\Gamma)p$ and $u_2\in L(\Sigma)'\cap L(\Gamma)$ such that 
\begin{equation*}
u_1 w^*w u_1^*=q_1\in Qb \quad \text{and} \quad u^*_2 ww^* u_2=q_2\in L(\Omega).
\end{equation*}
We denote $t:=u_2^*uu^*_1$ then the relation (\ref{6}) can be rewritten as
\begin{equation}\label{8}
t q_1 (U'\cap pL(\Gamma)p) q_1 t^*=q_2 (L(\Sigma)'\cap L(\Gamma))q_2. 
\end{equation}  
Since by (\ref{2eqs2}) $[\Gamma:\Sigma\Omega]<\infty$, we obtain that
\begin{equation*}
q_2 L(\Sigma \Omega) q_2 \subset q_2L(\Gamma)q_2
\end{equation*}
has finite index. Since $L(\Omega)\subset L(\Sigma)'\cap L(\Gamma)$, it follows that
\begin{equation*}
q_2 L(\Sigma \Omega)q_2\subset q_2\big(L(\Sigma )\vee (L(\Sigma)'\cap L(\Gamma))\big)q_2
\end{equation*}
is a finite index inclusion. Therefore following the same argument as the previous claim, we obtain that
\[
q_2L(\Omega)q_2\subset q_2L(\Sigma)'\cap L(\Gamma)q_2
\]
is a finite index inclusion of II$_1$ factors. By Lemma \ref{finiteindeximage}, there exist projections $r_1, r_2\leq q_2$, a partial isometry $w_1\in q_2L(\Sigma)'\cap L(\Gamma)q_2$,  and a $*$-isomorphism
\[
\phi': r_1L(\Sigma)'\cap L(\Gamma)r_1 \rightarrow B\subset r_2L(\Omega)r_2
\]
such that 
\begin{enumerate}
\item [(3)] \label{10'}$\phi'(x)w_1=w_1 x$ for all $x\in r_1L(\Sigma)'\cap L(\Gamma)r_1$;
\item [(4)]\label{10} $[r_2L(\Omega)r_2:B]<\infty$.
\end{enumerate}
Using Lemma \ref{intinsubfactor}, relation (\ref{8}), and perturbing more the unitary $t$, we can assume there exists a projection $q_3\in Q$ such that $q_3 b\leq q_1$ and 
\begin{equation}\label{8'}
t q_3 (U'\cap pL(\Gamma)p) q_3 b t^*=r_1 (L(\Sigma)'\cap L(\Gamma))r_1. 
\end{equation} 
Consider the $\ast$-isomorphism 
$\Psi':q_3Qq_3\rightarrow tq_3 Qq_3b t^*$
given by
\[
\Psi'(x)=txb t^*\quad\text{for} \quad x\in q_3Qq_3
\]
and we set $\Psi=\phi'\circ\Psi': q_3Qq_3 \rightarrow r_2L(\Omega)r_2.$ Clearly $\Psi$ is a $*$-homomorphis.
Using (3) above for every $x\in q_3Qq_3$ we have
\begin{align*}
\Psi(x) w_1
&=\phi'(\Psi'(x))w_1t = w_1 \Psi'(x) t=w_1txbt^*t\\
&= w_1 txb = w_1tb x.
\end{align*}
 
Next we will show that $w_1tb\neq 0$. Indeed, suppose by contradiction that  $w_1tb= 0$ then  $w_1tbq_1 t^* = 0$. This implies that $w_1q_2 = 0$. Thus
\begin{equation*}
w_1=w_1 r_1= w_1r_1q_2=0,
\end{equation*}
a contradiction.  So letting $\hat w$ to be the partial isometry in the polar decomposition of $w_1tb=\hat w|w_1tb|$, simply denoting
$q:=q_3$ and $f:=r_2$, we get that 
\begin{equation*}
\Psi: qQq\rightarrow fL(\Omega)f
\end{equation*} 
is an injective, unital $\ast$-homomorphism so that
\begin{equation*}
\Psi(x)\hat w=\hat wx \quad\text{for all}\quad x\in qQq.
\end{equation*}
Moreover since $Qb\subset q_1(U'\cap pL(\Gamma)p) q_1$ is finite index, using (4) above and (\ref{8'}) one gets that 
\[
\Psi(qQq)\subset r_2 L(\Omega)r_2
\]
has finite index. Altogether these show that 
$Q\cong^{com}_{L(\Gamma)} L(\Omega)$
as desired.  
\end{proof}


We end this section presenting the second main result. This roughly asserts that tensor product decompositions of group von Neumann algebras whose factors are commensurable with subalgebras arising commuting subgroups can be ``slightly perturbed'' to tensor product decompositions arising from the actual direct product decompositions of the underlying group. The proof uses the factor framework in an essential way and it is based on arguments from \cite[Proposition 12]{OP03} and \cite[Theorem 4.14]{CdSS15} (see also \cite[Theorem 6.1]{DHI16}).  

\begin{thm}\label{fromvirtualtoproduct}
Let $\Gamma$ be an icc group and assume that  $M=L(\Gamma)=M_1\bar\otimes M_2$, where $M_i$ are diffuse factors. Also assume there exist commuting, non-amenable, icc subgroups $\Sigma_1,\Sigma_2<\Gamma$ such that
\begin{center}
$[\Gamma:\Sigma_1 \Sigma_2]<\infty$, \quad $M_1\cong^{com}_M L(\Sigma_1)$, 
\quad and \quad $M_2\cong^{com}_M L(\Sigma_2)$.
\end{center}
Then there exist a group decomposition $\Gamma=\Gamma_1\times \Gamma_2$, a unitary $u\in M$ and $t>0$ such that 
\begin{center}
$M_1=uL(\Gamma_1)^tu^*$\quad and \quad $M_2=uL(\Gamma_2)^{1/t}u^*$. 
\end{center} 
 \end{thm}

\begin{proof}
Since $M_1\cong^{com}_M L(\Sigma_1)$, in particular we have $L(\Sigma_1)\prec_M M_1 $. Since $M=M_1\bar\otimes M_2$ then proceeding as in the proof of  \cite[Proposition 12]{OP03}  there exist a scalar $\mu>0$ and a partial isometry $v\in M$ satisfying 
\begin{center}
$p:=vv^*\in M^{1/\mu}_2$, \quad $q:=v^*v\in L(\Sigma_1)'\cap M$ \quad and
\end{center}
\begin{equation}\label{4.1'}
vL(\Sigma_1) v^*\subset M^\mu_1p.
\end{equation} 

Let $\Omega_2=\{ \gamma \in\Gamma \,|\, |\mathcal{O}_{\Sigma_1}(\gamma) |<\infty \}$ 
where $\mathcal{O}_{\Sigma_1}(\gamma)=\{\eta\gamma\eta^{-1}\,|\,\eta\in\Sigma_1\}$ is the orbit of $\gamma\in \Gamma$ under the conjugate action of $\Sigma_1$. Notice that for $\gamma_1, \gamma_2\in \Gamma$, it is easy to show that $\mathcal{O}_{\Sigma_1}(\gamma_1\gamma_2)\subset \mathcal{O}_{\Sigma_1}(\gamma_1)\mathcal{O}_{\Sigma_1}(\gamma_2)$. This implies that 
\begin{center}
$|\mathcal{O}_{\Sigma_1}(\gamma_1\gamma_2)|\leq |\mathcal{O}_{\Sigma_1}(\gamma_1)||\mathcal{O}_{\Sigma_1}(\gamma_2)|$ 
\end{center}
and hence $\Omega_2$ is a subgroup of $\Gamma$.
Clearly $ \Sigma_2 < \Omega_2$ because from the assumption $\Sigma_1, \Sigma_2$ commute.
Since $[\Gamma:\Sigma_1\Sigma_2]<\infty$, it follows that $[\Gamma: \Omega_2\Sigma_1]<\infty$. 

Now setting $\Omega_1=\operatorname{C}_{\Sigma_1}(\Omega_2)$, the centerizer of $\Omega_2$ in $\Sigma_1$, we can easily see that $\Omega_1,\Omega_2<\Gamma$ are commuting, non-amenable, icc subgroups.

\noindent \textit{Claim:} $[\Sigma_1:\Omega_1]<\infty$ and $[\Gamma:\Omega_1 \Omega_2]<\infty$.\\
First, we will show that $[\Sigma_1:\Omega_1]<\infty$. Assume by a contradiction that $\Omega_1$ has infinite index in $\Sigma_1$ and $\{h_k\}\subset \Sigma_1$ is an infinite sequence of representatives of distinct right cosets of $\Omega_1$ in $\Sigma_1$. Since $[\Gamma:\Omega_2\Sigma_1]<\infty$, there is a right coset $\Omega_2\Sigma_1\gamma$ such that $\Omega_2\Sigma_1\gamma\cap \Omega_1h_i\neq \emptyset$. Then consider the subsequence $\{h_i\}$ of $\{h_k\}$, for each $i\geq 1$, $h_i=x_i\gamma$ for some $x_i\in \Omega_2\Sigma_1$. Then 
\begin{center}
$h_ih_1^{-1}=(x_i\gamma)(x_1\gamma)^{-1}=x_i\gamma\gamma^{-1}x_1^{-1}=x_ix_1^{-1}\in \Omega_2\Sigma_1$
\end{center}
for all $i\geq 2$. Then for each $i\geq 2$, It follows that $h_ih_1^{-1}=\omega_i\sigma_i$ for some $\omega_i\in \Omega_2$ and $\sigma_i\in \Sigma_1$. Notice that $\omega_j\sigma_j=h_jh_1^{-1}\neq h_ih_1^{-1}=\omega_i\sigma_i$. From the construction, we have $\omega_i\Sigma_1\neq \omega_j\Sigma_1$ for all $i\neq j$ It is easy to check that $|\mathcal{O}_{\Omega_2\Sigma_1}(\omega_i)|<\infty$ for all $i\geq 2$. Also since $[\Gamma:\Omega_2\Sigma_1]<\infty$, it implies further that $|\mathcal{O}_{\Gamma}(\omega_i)|<\infty$. However, it contradicts the assumption $\Gamma$ is icc. Hence  $[\Sigma_1:\Omega_1]<\infty$.
Furthermore, as a consequence $[\Gamma:\Omega_1\Omega_2]<\infty$ as well.

Also notice that since ${C}_{\Gamma}(\Sigma_1)\subset \Omega_2,$ we have $L(\Sigma_1) '\cap M \subset L(\Omega_2)$ and  by relation (\ref{4.1'}) we have 
\[
vL(\Omega_1) v^*\subset M^\mu_1p.
\] 
Since $L(\Omega_2)$ and $M^{1/\mu}_2$ are factors then as in the proof of \cite[Proposition 12]{OP03}, we can find partial isomoetries 
\begin{center}
$w_1, \ldots,w_m\in L(\Omega_2)$\quad and  \quad $u_1,\ldots, u_m \in M^{1/\mu}_2$ 
\end{center}
satisfying 
\begin{center}
$w_i{w_i}^*=q'\leq q$,\quad ${u_i}^*u_i=p'=uq'u^*\leq p$ for any $i$ \quad and
\end{center}
\begin{center}
$\sum_j {w_j}^*w_j=1_{L(\Omega_2)}$, \quad $\sum_j u_j{u_j}^*=1_{M^{1/\mu}_2}$. 
\end{center}
Combining with the above, we can check that $u=\sum_j u_j vw_j\in M$ is a unitary satisfying $u L(\Omega_1)u^*\subset M^\mu_1$. Since we know that $M=M_1^{\mu}\bar\otimes M_2^{1/\mu}$, it is forced that
\begin{equation}\label{4.2'} 
M^{1/\mu}_2\subset u(L(\Omega_1)'\cap M)u^* .
\end{equation}

Similarly, let $\Theta_2=\{ \lambda \in\Gamma \,|\, |\mathcal O_{\Omega_1}(\lambda) |<\infty \}$ and $\Theta_1=C_{\Omega_1}(\Theta_2)$. As before it follows that $\Theta_1,\Theta_2 <\Lambda$ are commuting, non-amenable and icc subgroups  such that 
\begin{center}
$[\Gamma:\Theta_1 \Theta_2]<\infty$ \quad and\quad $[\Sigma_1:\Theta_1]<\infty$.  
\end{center}
Moreover, Since $C_{\Gamma}(\Omega_1)\subset \Theta_2$, by (\ref{4.2'}) we have
\[
M^{1/\mu}_2 \subset uL(\Theta_2)u^*.
\]
Since $M=M^\mu_1\bar \otimes M^{1/\mu}_2$, by Theorem \ref{Ge96theoremA}, there exists a subfactor $B\subset M^\mu_1$ such that
\[
uL(\Theta_2)u^*=B\bar \otimes M^{1/\mu}_2.
\]
Since $M_2\cong^{com}_M L(\Sigma_2),$ we have 
$uL(\Sigma_2)u^* \prec_M M^{1/\mu}_2.$
Since $[\Omega_2:\Sigma_2]<\infty$, it follows that 
$uL(\Omega_2)u^*\prec_M M_2^{1/\mu}$ as well.
 Since $B\subset uL(\Omega_2)u^*$ we have that $B\prec_M M_2^{1/\mu}$. 
 However since $B\subset M_1^{\mu}$ and $M=M_1^{\mu}\bar\otimes M_2^{1/\mu}$, these force that $B$ has an atomic corner. As $B$ is a factor, then we get
\begin{center}
$B=\mathbb{M}_k(\mathbb C)$, \quad for some\quad $k\in \mathbb N$. 
\end{center}
Altogether, we have 
\begin{equation}\label{K1}
uL(\Theta_2)u^*= B\bar\otimes M^{1/\mu}_2=\mathbb{M}_k(\mathbb C)\bar\otimes M^{1/\mu}_2=M^{t}_2,
\end{equation} 
where $t=k/\mu$. Since $M=M^{1/t}_1\bar\otimes M^t_2$, we also get 
\begin{equation}\label{K2}
u(L(\Theta_2)'\cap M)u^*= M^{1/t}_1.
\end{equation}
Let $\Gamma_1=\{ \lambda \in\Gamma \,|\, |\mathcal O_{\Theta_2}(\lambda) |<\infty \}$ and since $\Theta_2$ is an icc group, it follows that $\Gamma_1\cap\Theta_2=\{1\}$.  By construction as $C_{\Gamma}(\Theta_2)\subset \Gamma_1$, we obtain $uL(\Gamma_1)u^*\supseteq u(L(\Theta_2)'\cap M)u^*= M^{1/t}_1$. Therefore, agian
applying Theorem \ref{Ge96theoremA}, we have that
\begin{center}
$uL(\Gamma_1)u^*=A\bar\otimes M_1^{1/t}=A\bar \otimes u(L(\Theta_2)'\cap M)u^*$,
\end{center}
for some subfactor $A \subset uL(\Theta_2)u^*$.

In particular, we have $A=u L(\Gamma_1)u^*\cap u L(\Theta_2)u^* = \mathbb C 1$ since $ \Gamma_1\cap\Theta_2=\{1\}$ and, hence $uL(\Gamma_1)u^*=u(L(\Theta_2)'\cap M) u^*$. Letting $\Gamma_2=\Theta_2$, it follows that the subgroups $\Gamma_1$ and $\Gamma_2$ are commuting, non-amenable subgroups of $\Gamma$ such that $\Gamma_1\cap\Gamma_2=\{1\}$, $\Gamma_1\Gamma_2=\Gamma$. And from equation (\ref{K1}) and (\ref{K2}) above,  $uL(\Gamma_1)u^*=  M^{1/t}_1$, and $uL(\Gamma_2)u^*= M^{t}_2$.\end{proof}


\section{Classification of tensor product decompositions of II$_1$ factors arising from groups}

Motivated by the prior work \cite{CdSS15}, Drimbe, Hoff and Ioana have discovered in \cite{DHI16} a new classification result in the study of tensor product decompositions of II$_1$ factors. Specifically they unveiled the first examples of icc groups $\Gamma$ for which all diffuse tensor product decompositions of $L(\Gamma)$ are ``paramatrized'' by the canonical direct product decompositions of the underlying group $\Gamma$. Their examples include remarkable groups such as the class of all icc groups $\Gamma$ that are measure equivalent to products of non-elementary hyperbolic groups. Similar results where obtained subsequently in \cite{CdSS17,dSP17}. In this dissertation we obtained similar results for new classes of groups including amalgamated free products, direct products of wreath product groups and MsDuff's groups. For the ease of presentation the results will be presented in independent subsections.


\subsection{Amalgamated free product groups} 

In Section \ref{AFPtensor} we have seen that for a large class of AFP von Neumann algebras $M=M_1\ast_P M_2$ all their tensor factorizations essentially split P and the entire inclusions $P\subset M_i$. However in the particular case when $M$ arises for amalgam groups $\Gamma=\Gamma_1\ast_\Sigma \Gamma_2$ this is insufficient to determine whether this further splits the group $\Sigma$ as well. In fact it is well known this does not happen all the time (see the \emph{Remark} after the Theorem \ref{tensordecompamalgam}) and hence a separate analysis is required to understand this aspect. In this direction we isolate several situations when indeed the tensor decompositions arise from the direct product splittings of $\Gamma$. One instance is when the algebra $L(\Sigma)$ is \textit{virtually prime}\footnote{See Definition \ref{virtuallyprimedef}.}.  

Before stating our result we need a group theoretic preliminary.

\begin{lem}\label{amalgamprodecomp} 
Let $\Gamma=\Gamma_1\ast_\Sigma \Gamma_2$ be an amalgamated free product. Suppose $\Gamma=\Lambda_1\times \Lambda_2$ for some subgroups $\Lambda_1, \Lambda_2$. Then we can find a permutation $\sigma\in \mathfrak{S}_2$  satisfying
\begin{itemize}
\item $\Sigma=\Lambda_{\sigma(1)} \times \Sigma_0$, 
\item $\Gamma_1=\Lambda_{\sigma(1)} \times \Gamma^0_1$, 
\item $\Gamma_2= \Lambda_{\sigma(1)}\times \Gamma^0_2$, 
\item $\Lambda_{\sigma(2)}=\Gamma^0_1\ast_{\Sigma_0} \Gamma^0_2$.
\end{itemize}
\end{lem}
\begin{proof}
Since $\Gamma=\Gamma_1\ast_\Sigma \Gamma_2$, considering the von Neumann algebra of $\Gamma$, we have $L(\Gamma)=L(\Gamma_1)\ast_{L(\Sigma)} L(\Gamma_2)$.
By using Theorem \ref{intertwiningincore1}, we have 
\begin{center}
$L(\Gamma_{\sigma(1)})\prec_{L(\Gamma)} L(\Sigma)$ 
\end{center}
for some $\sigma\in \mathfrak{S}_2$.
Since $\Gamma=\Lambda_1\times \Lambda_2$, by applying Theorem \ref{CI17lemma2.2}, there is an element $h\in \Gamma$ so that
$[\Lambda_{\sigma(1)}: h\Sigma h^{-1}\cap\Lambda_{\sigma(1)}]<\infty.$
Since $\Lambda_{\sigma(1)}$ is normal in $\Gamma$, conjugating by $h$ we can assume that $[\Lambda_{\sigma(1)}: \Sigma \cap \Lambda_{\sigma(1)}]<\infty$. Also passing through a finite index subgroup, we can also assume  $\Sigma \cap \Lambda_{\sigma(1)}$ is normal in $\Gamma$. Therefore, we have 
\begin{align*}
\Gamma\big/(\Sigma\cap\Lambda_{\sigma(1)})
&= (\Gamma_1\big/(\Sigma\cap \Lambda_{\sigma(1)})) \ast_{\Sigma \big/(\Sigma\cap \Lambda_{\sigma(1)})} (\Gamma_2 \big/(\Sigma\cap \Lambda_{\sigma(1)})) \\
&= \Lambda_{\sigma(1)}\big/(\Sigma \cap \Lambda_{\sigma(1)}) \times \Lambda_{\sigma(2)}.
\end{align*}
Since $\Lambda_{\sigma(1)}\big/(\Sigma \cap\Lambda_{\sigma(1)})$ is finite, \cite[Theorem 10]{KS70}  implies that 
\[
\Lambda_{\sigma(1)}\big/(\Sigma \cap \Lambda_{\sigma(1)})<\Sigma \big/(\Sigma\cap\Lambda_{\sigma(1)})
\]
and thus $\Lambda_{\sigma(1)}< \Sigma $. Since $\Lambda_{\sigma(1)}<\Sigma \subset \Lambda_1\times \Lambda_2$ and clearly $\Lambda_{\sigma(1)}$ is normal in $\Lambda_1\times \Lambda_2$, 
there is a subgroup $\Sigma_0$ of $\Sigma$ such that  $\Sigma=\Lambda_{\sigma(1)}\times \Sigma_0$. With the same argument, since $\Sigma< \Gamma_1, \Gamma_2 < \Lambda_1\times \Lambda_2$, for $i={1,2}$ there are subgroups $\Gamma^0_i<\Gamma_i$ such that
$\Gamma_i= \Lambda_{\sigma(1)} \times \Gamma^0_i$. Moreover, 
\begin{align*}
\Lambda_{\sigma(1)}\times \Lambda_{\sigma(2)}
&=\Gamma=\Gamma_1*_{\Sigma}\Gamma_2\\
&=(\Lambda_{\sigma(1)} \times \Gamma^0_1)\ast_{(\Lambda_{\sigma(1)}\times (\Sigma_0)} \Lambda_{\sigma(1)} \times \Gamma^0_2)\\
&=\Lambda_{\sigma(1)}\times (\Gamma^0_1\ast_{\Sigma_0} \Gamma^0_2).
\end{align*}
Hence, we can conclude that $\Lambda_{\sigma(2)}=\Gamma^0_1\ast_{\Sigma_0} \Gamma^0_2$. 
\end{proof}

\begin{thm}\label{tensordecompamalgam}
 Let $\Gamma=\Gamma_1\ast_{\Sigma}\Gamma_2$ be an icc group with $[\Gamma_1:\Sigma]\geq 2$ and $[\Gamma_2:\Sigma]\geq 3$. Assume that $\Sigma$ is finite-by-icc
and any corner of $L(\Sigma)$ is virtually prime. Suppose that $L(\Gamma)=M_1\bar\otimes M_2$ for diffuse $M_i$'s. Then there exist direct product decompositions 
\begin{equation*}
\Sigma=\Omega \times \Sigma_0,\,\,\,  \Gamma_1= \Omega \times \Gamma^0_1, \,\,\,\text{and} \,\,\,\Gamma_2= \Omega \times \Gamma^0_2
\end{equation*}
with  $\Sigma_0$ finite, for some groups $\Sigma_0<\Gamma^0_1, \Gamma^0_2$, and hence $\Gamma=\Sigma\times (\Gamma_1^0\ast_{\Sigma_0}\Gamma_2^0)$. Moreover, there exist a unitary $u\in L(\Gamma)$, a scalar $t>0$ and $\sigma\in\mathfrak S_2$ such that  
\begin{equation*}
M_{\sigma(1)}= u L(\Omega)^t u^*\quad \text{ and }\quad  M_{\sigma(2)} = u L(\Gamma_1^0\ast_{\Sigma_0} \Gamma_2^0)^{1/t} u^*.
\end{equation*}
\end{thm}

\begin{proof}Since $M_1\bar\otimes M_2=L(\Gamma)$, by Corollary \ref{intertwiningcoregroupsafphnn} we can assume $M_{\sigma(1)} \prec L(\Sigma)$. Since any corner  of $L(\Sigma)$ is virtually prime then by Lemma \ref{intertwiningdichotomy} we must have  
\[
M_{\sigma(1)}\cong_M^{com} L(\Sigma),
\]
and further applying Theorems \ref{virtualprod} and \ref{fromvirtualtoproduct} there exist infinite groups $\Lambda_i$ so that $\Gamma=\Lambda_1\times \Lambda_2$. Thus the desired conclusion follows by using Lemma \ref{amalgamprodecomp}. \end{proof}

\begin{remark}The previous theorem illustrates a situation when a true von Neumann algebraic counterpart of Lemma \ref{amalgamprodecomp} could be successfully obtained. However, if one drops the primeness assumption on $L(\Sigma)$, the conclusion of the theorem is no longer true. Precisely, there are icc amalgams $\Gamma=\Gamma_1\ast_\Sigma \Gamma_2$ whose group factors $L(\Gamma)$ admit non-canonical tensor product decompositions while $\Gamma$ is indecomposable as a nontrivial direct product. For instance, consider a group inclusion $\Sigma <\Omega$ satisfying the following conditions: 
\begin{enumerate} 
\item [i)] \cite{Jo98} for each finite $E\subset \Omega$ there are $\gamma,\lambda \in \Sigma$ so that 
\begin{center}
$[\gamma,E]=[\lambda,E]=1$ \quad and \quad $[\gamma,\lambda]\neq 1$;
\end{center}
\item [ii)] for each $\gamma\in \Sigma$ there is $\lambda\in \Omega$ so that $[\gamma,\lambda]\neq 1$.
\end{enumerate}
Concrete such examples are $\Sigma = \oplus_{\mathfrak S_{\infty}} H < \Omega= \cup_{n\in \mathbb N}(H\wr \mathfrak S_n)$, where $H$ is any icc group and $\mathfrak S_{\infty}$ is the group of finite permutations of $\mathbb N$.

Then the inclusion  $\Sigma <\Gamma=\Omega\ast _\Sigma \Omega$ still satisfies i) and by \cite[Proposition 2.4]{Jo98} $L(\Gamma)$ is McDuff so $L(\Gamma)=L(\Gamma)\bar\otimes \mathcal{R}$, where $\mathcal{R}$ is the hyperfinite factor. On the other hand, combining Lemma \ref{amalgamprodecomp} with ii) one can see that $\Gamma$ cannot be written as a nontrivial direct product. 
\end{remark}


\subsection{Direct product of wreath product groups}\label{wreathproductdecompostion}

Throughout this section, we denote by $\mathcal{WR}$, the class of generalized wreath product groups in the form $\Gamma= A\wr_I G$,
where $G$ is a group acting on a set $I$, $A$ is an amenable group whose stabilizers $\operatorname{Stab}_{\Gamma}(i)$ are finite for all $i\in I$.

For further use we recall the following result, which is a particular case of \cite[Corollary 4.3]{IPV10}.

\begin{thm}[\cite{IPV10}]\label{AA1}
Let $\Gamma=A\wr_{I}\Gamma_0 \in \mathcal{WR}$ and let $B$ be a finite von Neumann algebra $B$. Denote by $M=B  \bar\otimes L(\Gamma)$ the corresponding tensor product algebra. Let $P_1, P_2\in pMp$ be two commuting von Neumann subalgebras such that $P_1\vee P_2\subset pMp$ is a finite index inclusion, Then either 
\begin{enumerate}
\item [i)] there exists a nonzero $p_0\in P_1'\cap pMp$ such that $P_1p_0$ is amenable relative to $B$ or 
\item [ii)] $P_2\prec_M B$ 
\end{enumerate}
\end{thm}
\begin{proof} Apply \cite[Corollary 4.3]{IPV10}, one of the following must hold:
\begin{enumerate}
\item[(1)] There exists $p_1\in (P_1)'\cap M$ such that $(P_1 )p_1$ is amenable relative to $B$ inside $M$;
\item[(2)] $P_2\prec_{M} B$;
\item[(3)] $P_1\vee (P_1'\cap pMp)\prec_{M} M\bar\otimes L(A^I)$.
\end{enumerate}

To finish the proof we only need to show that (3) does not hold. Assuming by contradiction it holds then, since $P_2\subset P_1'\cap pMp$, we have
$P_1\vee P_2\prec_{M} B\bar\otimes L(A^I)$. Together with the assumption that $P_1\vee P_2\subset pMp$ has finite index, these imply that $pMp \prec_{M} B\bar\otimes L(A^I)$. This further implies that $ B\prec_B L(A^I)$ which is a contradiction. \end{proof}

\noindent \textbf{Notation.} Let $\Gamma_1, \Gamma_2, \ldots, \Gamma_n$ be groups and let $\Gamma=\Gamma_1\times\Gamma_2\times \cdots \times \Gamma_n$ the corresponding $n$-folded direct product. For every subset $I\subset \{1, 2,\ldots, n\}$ we will be denoting by $\Gamma_I <\Gamma$ the subproduct groups supported on $I$, i.e.\ $\Gamma_I=\Pi_{i\in I}\Gamma_i$.


Next we present the main result of the section which classify all tensor product decompositions of II$_1$ factors associated with $n$-folded products of wreath product groups. In particular our result generalizes the unique prime decompositions results for such factors obtained by Sizemore and Winchester \cite{SW11}.

\begin{thm}
Let $\Gamma_1,\Gamma_2, \ldots, \Gamma_n\in \mathcal{WR}$ and let $\Gamma=\Gamma_1\times\Gamma_2\times \cdots \times \Gamma_n$. Consider the corresponding von Neumann algebra $M=L(\Gamma)$ and let $P_1, P_2$ be non-amenable II$_1$ factors such that $M=P_1\bar\otimes P_2$. Then there exist a scalar $t>0$ and a partition $I_1\sqcup I_2=\{1,2,\ldots,n\}$ such that
\begin{equation*}
L(\Gamma_{I_1})\cong P_1^t \quad \text{and} \quad L(\Gamma_{I_2})\cong P_2^{1/t}.
\end{equation*}
\end{thm}

\begin{proof}
Pick $I_1, I_2\subset \{1,2,\ldots,n\}$ be minimal (nonempty) subsets so that $P_1\prec_M L(\Gamma_{I_1})$ and $P_2\prec_M L(\Gamma_{I_2})$. Next we argue that $I_1\subsetneqq \{1,2,\ldots,n\}$ and $I_2\subsetneqq \{1,2,\ldots,n\}$.
We will only show the first statement as the second will follow similarly. Fix $i\in \{1,2,\ldots,n\}$. 
Write $M=L(\hat\Gamma_i)\otimes L(\Gamma_i)$ where $\hat\Gamma_i:= \Gamma_{\{1,...,n\}\setminus\{i\}}$ and using Theorem \ref{AA1}  for $B_i=L(\hat\Gamma_i)$  we have that either
\begin{enumerate}
\item [(a)]$P_1\prec_M L(\hat\Gamma_i)$ or
\item [(b)]$P_2\otimes p_i$ is amenable relative to $L(\hat\Gamma_i)$ inside $M$ for some nonzero projection $p_i\in P_1$. 
\end{enumerate}
Notice that using Lemma \ref{lemma2.6DHI16} (2), since $P_2$ is a factor, case (b) above is equivalent to 
\begin{enumerate}
\item [(b')]$P_2$ is amenable relative to $L(\hat\Gamma_i)$ inside $M$.
\end{enumerate}
 Assume by contradiction that for all $i\in \{1,\ldots,n\}$ we have only case (b'). Since $E_{L(\hat\Gamma_i)}\circ E_{L(\hat\Gamma_j)}=E_{L(\hat\Gamma_j)}\circ E_{L(\hat\Gamma_i)}$ for all $i,j$ and $L(\hat\Gamma_j)\subset M$ is regular, by using Proposition \ref{proposition2.7PV11} inductively we have that $P_2$ is amenable relative to  $\bigcap_{i=1}^n L(\hat\Gamma_i)=\mathbb{C}1$ inside $M$. In particular, this implies that $P_2$ is amenable which  contradicts the initial assumption. Therefore, there exists an $i_o\in\{1,\ldots,n\}$ such that $P_1\prec_M L(\hat\Gamma_{i_0})$. In particular this show that $I_1\subset \{1,\ldots,n\}\setminus\{i_0\}$. Similarly we have that $I_2\subsetneqq\{1,\ldots,n\}$.

Next we prove the following
\begin{equation}\label{com1}P_1\cong^{com}_M L(\Gamma_{I_1}).
\end{equation}

To see this recall that $P_1\prec_M L(\Gamma_{I_1})$. Since $P_1\vee P_2=M$ and $\Gamma_{I_1}$ is icc, by using Lemma \ref{intertwiningdichotomy} one of the followings must hold:
\begin{enumerate}
\item[(a)] $P_1\cong^{com}_M L(\Gamma_{I_1})$, or 
\item[(b)] there exist nonzero projections $p_1\in P_1$, $q_1\in L(\Gamma_{I_1})$, a nonzero partial isometry $v\in q_1Mp_1$, and a $*$-isomorphism $\psi:p_1P_1p_1\rightarrow Q\subset q_1L(\Gamma_{I_1})q_1$ such that 
\begin{enumerate}
\item[(i)] $\psi(x)v=vx$ for $x\in p_1P_1p_1$;
\item[(ii)] $Q$ and $Q'\cap q_1L(\Gamma_{I_1})q_1$ are II$_1$ factors so that 
$Q\vee(Q'\cap q_1L(\Gamma_{I_1})q_1)\subset q_1L(\Gamma_{I_1})q_1$ has finite index;
\item[(iii)] $\operatorname{s}(E_{L(\Gamma_{I_1})}(vv^*))=q_1$.
\end{enumerate}
\end{enumerate}
So to show \eqref{com1} we only need to argue that the case (b) above does not hold. Assume by contradiction it does. As it is well-known that  the algebras $L(\Gamma_i)$ are prime for all $i\in \{1,\ldots,n\}$ (see for instance \cite[6.4]{Po07}), the part (ii) above implies that $|I_1|\geq 2$. Fix $j\in I_1$. 
From (ii) we have that $Q\vee(Q'\cap q_1L(\Gamma_{I_1})q_1)\subset q_1L(\Gamma_{I_1})q_1$ has finite index, and hence using Theorem \ref{AA1} we have that either
\begin{enumerate}
\item[(c)] $Q\prec_{q_1L(\Gamma_{I_1})q_1}L(\Gamma_{I_1\setminus\{j\}})$, or
\item[(d)] there exists a nonzero projection $p_0\in (Q'\cap q_1L(\Gamma_{I_1})q_1)'\cap q_1L(\Gamma_{I_1})q_1$ such that
$(Q'\cap q_1L(\Gamma_{I_1})q_1)p_0$ is amenable relative to $L(\Gamma_{I_1\setminus\{j\}})$ inside $L(\Gamma_{I_1})$.
\end{enumerate}
Since $Q\vee (Q'\cap q_1L(\Gamma_{I_1})q_1)$ is a factor, one can easily see that the inclusion $Q\vee Q'\cap q_1L(\Gamma_{I_1})q_1\subset q_1L(\Gamma_{I_1})q_1$ is \textit{irreducible}
\footnote{A subfactor of finite index $N\subset M$ is said to be \textit{irreducible} if 
the relative commutant $N'\cap M=\mathbb{C}$.}; 
in particular the normalizer satisfies that
\begin{equation*}
\big(\mathcal{N}_{q_1L(\Gamma_{I_1})q_1}(Q'\cap q_1L(\Gamma_{I_1})q_1)\big)'\cap q_1L(\Gamma_{I_1})q_1=\mathbb{C}1.
\end{equation*}
Hence, using Lemma \ref{lemma2.6DHI16} we see that the condition (d) is equivalent to 
\begin{enumerate}
\item[(d')] $Q'\cap q_1L(\Gamma_{I_1})q_1$ is amenable relative to $L(\Gamma_{I_1\setminus\{j\}})$ inside $L(\Gamma_{I_1})$.
\end{enumerate}
Assume that for every $j\in I_1$ only the possibility (d') holds. Since 
$E_{L(\Gamma_{I_1\setminus\{j_1\}})}\circ E_{L(\Gamma_{I_1\setminus\{j_2\}})}=E_{L(\Gamma_{I_1\setminus\{j_2\}})}\circ E_{L(\Gamma_{I_1\setminus\{j_1\}})}$ for all $j_1,j_2\in I_1$ and $L(\Gamma_{I_1\setminus\{j\}})$ are regular in $L(\Gamma_{I_1})$ then applying Proposition \ref{proposition2.7PV11}
 inductively we get that  $Q'\cap q_1L(\Gamma_{I_1})q_1$ is amenable relative to $\bigcap_{j\in I_1}L(\Gamma_{I_1\setminus\{j\}})=\mathbb{C}1$. 
It follows that $Q'\cap q_1L(\Gamma_{I_1})q_1$ is isomorphic to the hyperfinite II$_1$ factor. In particular, 
$Q\vee(Q'\cap q_1L(\Gamma_{I_1})q_1)$ is a factors with McDuff's property. In particular, it has property \emph{Gamma} of Murray-von Neumann.
Since $Q\vee(Q'\cap q_1L(\Gamma_{I_1})q_1)\subset q_1L(\Gamma_{I_1})q_1$ has finite index, it follows from \cite[Proposition 1.11]{PP86} that $q_1L(\Gamma_{I_1})q_1$ has property \emph{Gamma} as well. Therefore, for every $\omega$ non-principal ultrafilter on $\mathbb{N}$ we have that 
\begin{equation}\label{rel2}
L(\Gamma_{I_1})'\cap L(\Gamma_{I_1})^{\omega}\neq \mathbb{C}1.
\end{equation}
Thus $L(\Gamma_{I_1})$ has property \emph{Gamma}. Notice that $L(\Gamma_{I_1})=L(\Gamma_{I_1\setminus\{j\}})\bar\otimes L(\Gamma_j)$ and using both Example 1.4 c
\footnote{
Let $H, \Gamma$ be countably infinite discrete group, let $G\curvearrowright I$ , and consider the generalized wreath product group $H\wr_{I}\Gamma:=(\oplus_I H)\rtimes \Gamma.$. Let $\mathcal{G}:=\{\text{Stab}_{\Gamma}i\,|\, i\in I\}$. We have this group statisfies condition $\textbf{NC}$ with respect to $\mathcal{G}$
 }and Theorem 3.1\footnote{
 See Theorem \ref{CSU13theorem3.1}} in \cite{CSU13} we have that 
\begin{equation*}
L(\Gamma_{I_1})'\cap L(\Gamma_{I_1})^{\omega}\subset L(\Gamma_{I_1\setminus\{j\}})^{\omega}\vee L(\Gamma_j).
\end{equation*}
Since this holds for all $j\in I_1$ then we have that 
\begin{equation*}
L(\Gamma_{I_1})'\cap L(\Gamma_{I_1})^{\omega}\subset \bigcap_{j\in I_1} \big(L(\Gamma_{I_1\setminus\{j\}})^{\omega}\vee L(\Gamma_j)\big).
\end{equation*}
But by using the same argument from \cite[Corollary 1.2]{CP10} one can check that $\bigcap_{j\in I_1} \big(L(\Gamma_{I_1\setminus\{j\}})^{\omega}\vee L(\Gamma_j)\big)=L(\Gamma_{I_1})$ and hence 
\begin{center}
$L(\Gamma_{I_1})'\cap L(\Gamma_{I_1})^{\omega}\subset L(\Gamma_{I_1})\cap L(\Gamma_{I_1})'=\mathbb{C}1$ 
\end{center}
which is a contradiction to (\ref{rel2}).
Thus there must exist $j_0\in I_1$ such that  $Q\prec_{L(\Gamma_{I_1})} L(\Gamma_{I_1\setminus\{j_0\}})$. It follows that there exists nonzero projections $r\in Q$, $t\in L(\Gamma_{I_1\setminus\{j_0\}})$
and a nonzero partial isometry $w\in tL(\Gamma_{I_1})r$ and an injective $*$-homomorphism $\Phi:rQr\rightarrow tL(\Gamma_{I_1\setminus\{j_0\}})t$ such that 
\begin{equation}\label{rel3}
\Phi(y)w=wy \quad \text{for}\quad y\in rQr.
\end{equation}
Since $\psi$  is an isomorphism, there is a nonzero projection $p_0\in P_1$ such that $\psi(p_0)=r$. Thus the relation (i) implies that 
\begin{equation}\label{rel4}
\psi(x)v=vx \quad \text{for}\quad x\in p_0P_1p_0.
\end{equation}
Applying (\ref{rel4}) in (\ref{rel3}), we see that for all $x\in P_1$ we have that
\begin{equation}\label{rel5}
\Phi(\psi(x))wv=w\psi(x)v=wvx.
\end{equation}
Next we argue that  \begin{equation}\label{nonzero1}wv\neq 0.\end{equation}
Assume by contradiction that $wv=0$. Thus $wvv^*=0$ end hence 
\begin{equation*}
0=E_{L(\Gamma_{I_1})}(wvv^*)=wE_{L(\Gamma_1)}(vv^*).
\end{equation*}
But this implies  that $0=w \operatorname{s}(E_{L(\Gamma_{I_1})}(vv^*))$, where $\operatorname{s}(E_{L(\Gamma_{I_1})}(vv^*))$ is the support projection of $E_{L(\Gamma_{I_1})}(vv^*)$. By (iii) we get $0=wq_1$ and since by construction $r\leq q_1$ and $w\in tL(\Gamma_{I_1})r$ then we get $wq_1=w$, hence $w=0$ which is a contradiction. This proves \eqref{nonzero1}.

Therefore $wv\neq 0$ and by taking the polar decomposition of $wv=w_0|wv|$, we see that (\ref{rel5}) implies
\begin{equation}
\Phi\circ\psi(x)w_0=w_0 x \quad \text{for all}\quad x\in p_0P_1p_0.
\end{equation}
Since $\Phi\circ\psi:p_0P_1p_0\rightarrow tL(\Gamma_{I_1\setminus\{j_0\}})t$ is  a $*$-homomorphism, it follows that $P_1\prec_{L(\Gamma_{I_1})} L(\Gamma_{I_1\setminus\{j_0\}})$ but this contradicts the minimality of $I_1$ and therefore we have reached a contradiction. As a consequence, case (b) does not hold altogether. 

Using relation \eqref{com1} and Theorem \ref{virtualprod} 
there exist a subgroup $\Omega\leqslant C_{\Gamma}(\Gamma_{I_1})=\Gamma_{I\setminus I_1}$ such that $\Omega\times \Gamma_{I_1}\leqslant \Gamma$ is finite index and $P\cong^{com}_M L(\Omega)$. Hence by Theorem \ref{fromvirtualtoproduct} we conclude that there exist $\Gamma_1\times \Gamma_2=\Gamma$ a product decomposition and a scalar $t>0$ and a unitary $u\in \mathcal{U}(M)$ such that
\begin{equation*}
L(\Gamma_1)=uP_1^tu^*\quad \text{and}\quad L(\Gamma_2)=uP_2^{1/t}u^*.
\end{equation*}
Moreover, it is implicit in the proof of Theorem \ref{fromvirtualtoproduct}  that $\Gamma_{I_1}$ is commensurable to $\Gamma_1$ and $\Gamma_{\{1,...,n\}\setminus I_1}=\Gamma_{I_2}$ is commensurable to $\Gamma_2$. It only remains to argue that $\Gamma_{I_1}=\Gamma_1$ and $\Gamma_{I_2}=\Gamma_2$ which follows from basic group theoretic considerations. 
\end{proof}


\subsection{McDuff's group functors $T_0$ and $T_1$}

In this subsection we establish tensor product decomposition results for II$_1$ factors associated with groups that arise via $T_0$, $T_1$-group functorial constructions introduced by D. McDuff in \cite{Mc69}. Before doing so we recall those notations from \cite{Mc69}. These constructions are inspired by the earlier work of Dixmier and Lance \cite{DL69} which in turn go back to the pioneering work of Murray and von Neumann \cite{MvN43}.

Let $\Gamma$  be a group. For $i\geq 1,$ let $\Gamma_i$ be isomorphic copies of $\Gamma$
 and $\Lambda_i$ be isomorphic to $\mathbb{Z}$. Define $\tilde\Gamma=\bigoplus_{i\geq 1} \Gamma_i$ and let $\mathfrak S_{\infty}$ be the group of finite permutations of the positive integers $\mathbb N$. Consider the semidirect product $\tilde\Gamma\rtimes \mathfrak S_{\infty}$ associated to the natural action of $\mathfrak S_{\infty}$ on $\tilde\Gamma$ which permutes the copies of $\Gamma$. Following \cite{Mc69} we define
\begin{itemize}
\item $T_0(\Gamma)$ = the group generated by $\tilde\Gamma$ and $\Lambda_i, i\geq 1$ with the only relation that $\Gamma_i$ and $\Lambda_j$ commutes for $i\geq j\geq 1$.
\item  $T_1(\Gamma)$ = the group generated by $\tilde\Gamma\rtimes \mathfrak S_{\infty}$ and $\Lambda_i, i\geq 1$ with the only relation that $\Gamma_i$ and $\Lambda_j$ commute for $i\geq j\geq 1$.
\end{itemize}


Using a basic iterative procedure, these famous functorial group constructions were used to provide the first infinite family of non-isomorphic II$_1$ factors, the so called $L(K_\alpha(\Gamma) )$'s where $\alpha\in \{0,1\}^{\mathbb N}$. One key feature, which also played a crucial role in McDuff's work, is that the corresponding group factors $L(T_\alpha (\Gamma))$ possess lots of central sequences. In particular these algebras have McDuff property, i.e. $L(T_\alpha(\Gamma))\cong L(T_\alpha(\Gamma))\bar\otimes\mathcal R$, where $\mathcal R$ is the hyperfinite II$_1$ factor. However we will prove below that these are the only possible tensor decompositions. Specifically we have the following type of unique prime factorization result

\begin{thm}\label{MDsplitting}
Fix $\Gamma$  a non-amenable group and let $\alpha\in\{0,1\}$. If $L(T_{\alpha}(\Gamma))=P_1\bar\otimes P_2$ then either $P_1$  or $P_2$ is isomorphic to the hyperfinite II$_1$ factor.
\end{thm}

\begin{proof}
First denote by $\tilde \Gamma_n:=\oplus_{i\geq n} \Gamma_i$.  Let $\alpha=0$ and define 
\begin{itemize}
\item $\Sigma_{n}\leqslant T_0(\Gamma)$ be the subgroup generated by $\tilde\Gamma, \Lambda_1, \Lambda_2,\ldots,\Lambda_n$; 
\item $\Delta_n\leqslant T_0(\Gamma)$ be the subgroup generated by $\tilde\Gamma_n, \Lambda_{n+1}, \Lambda_{n+2},\ldots$. 
\end{itemize}
Similarly in the case of $\alpha=1$, we define
\begin{itemize}
\item $\Sigma_n\leqslant T_1(\Gamma)$ is the subgroup generated by $\tilde\Gamma\rtimes \mathfrak S_{\infty}, \Lambda_1, \Lambda_2,\ldots,\Lambda_n$; 
\item $\Delta_n\leqslant T_1(\Gamma)$ is the subgroup generated by $\tilde\Gamma_n, \Lambda_{n+1}, \Lambda_{n+2}$. 
\end{itemize}
In both cases, one can check that
\begin{equation*}
T_{\alpha}(\Gamma)=\Sigma_n*_{\tilde\Gamma_n}\Delta_n. \quad\text{Thus,}\quad
L(T_{\alpha}(\Gamma))=L(\Sigma_n)*_{L(\tilde\Gamma_n)}L(\Delta_n).
\end{equation*}
And we denote by $\Sigma_n':=(\bigoplus_{i=1}^{n-1}\Gamma_i)\vee \Lambda_1\vee\Lambda_2\vee\cdots\vee\Lambda_n< \Sigma_n.$

Now let $M=L(T_{0}(\Gamma))=P_1\bar\otimes P_2.$ Then by Theorem \ref{intertwiningincore1} there exist $i\in\{1,2\}$ 
such that $P_i\prec_M L(\tilde\Gamma_n)$. Since $P_i$ are factor, we have 
\begin{equation}\label{BB6}
P_i\prec_M^s L(\tilde\Gamma_n). 
\end{equation}
Next denote by $Q_n:=L(\tilde\Gamma_n)$ and $M_n:=L(\Sigma_n')$. With these notations at hand we show the followings hold.
\begin{eqnarray}
&&\label{eqq1} \lim_{n\rightarrow\infty}\|x-E_{M_n}(x)\|_2=0\text{ for all }x\in M.\\
&&\label{eqq2} \text{as } Q_n-M_n\text{ bimodules we have }{}_{Qn}L^2(M)_{M_n}\prec \,{}_{Q_n}L^2(Q_n)\bar\otimes L^2(M)_{M_n}.
\end{eqnarray}

To justify these statements notice first, since $\Sigma_n':=(\bigoplus_{i=1}^{n-1}\Gamma_i)\vee \Lambda_1\vee\Lambda_2\vee\cdots\vee\Lambda_n$, then clearly $\Sigma_n'\nearrow \bigcup_{n\geq 1}\Sigma_n'=T_0(\Gamma) $ and hence $M=L(T_0(\Gamma))=\overline{\bigcup_{n}L(\Sigma_n')}^{\tiny{\text{SOT}}}=\overline{\bigcup_{n}M_n}^{\tiny{\text{SOT}}}$. This clearly shows \eqref{eqq1}. 

Now we show \eqref{eqq2}. As before we have that  $T_0(\Gamma)=\Sigma_n *_{\tilde\Gamma_n}\Delta_n$. Notice that $\Sigma_n=\Sigma_n'\times \tilde\Gamma_n.$ Fix $\mathcal{F}$ a set of left coset representatives for $\Sigma_n'$ in $\Gamma$ and we isolate the following subsets of $\mathcal{F}$:
\begin{align*}
\mathcal{F}_1&=\{w\,|\, w=a_1b_1a_2b_2\ldots a_kb_k\,\,\text{or}\,\, b_1a_2b_2\ldots a_kb_k \,\text{where}\, a_i\in \Sigma_n\setminus\tilde\Gamma_n, b_i\in \Delta_n\setminus\tilde\Gamma_n \};\\
\mathcal{F}_0&=\{w\,|\, w\in \tilde\Gamma_n\}.
\end{align*}
We can check that $\mathcal{F}_1\sqcup\mathcal{F}_0=\mathcal{F}$.

Next we prove that if $\tilde\Gamma_n w_1\Sigma_n'=\tilde\Gamma_n w_2\Sigma_n'$ for $w_1, w_2\in\mathcal{F}_1$, then $w_2^{-1}w_1\in \tilde\Gamma_n$. Indeed, let $m_1,m_2\in \Gamma_n', k_1,k_2\in \Sigma_n'$ such that 
\begin{center}
$m_1w_1k_1=m_2w_2k_2$.\quad Thus\quad$m_1w_1k_1k^{-1}_2w^{-1}_2m^{-1}_2=1.$
\end{center}
As $w_i={\ldots a^{(i)}_kb^{(i)}_k}$ where $a^{(i)}_k\in \Sigma_n\setminus\tilde\Gamma_n$ and $b^{(i)}_k\in \Delta_n\setminus \tilde\Gamma_n$, we see that the previous equation implies that 
\begin{equation}\label{BB1}
m_1\ldots b^{(1)}_{k-1}a_k^{(1)}b_k^{(1)}k_1k_2^{-1}(b_k^{(2)})^{-1}(a_k^{(2)})^{-1}(b_{k-1}^{(2)})^{-1}\ldots m_2^{-1}=1
\end{equation}
Consider the part $b_k^{(1)}k_1k_2^{-1}(b_k^{(2)})^{-1}$ and notice that if $k_1k_2^{-1}\neq 1$ then $k_1k_2^{-1}\in\Sigma_n'\setminus\{1\}\subset\Sigma_n' \setminus \tilde\Gamma_n$ because $\Sigma'_n\cap\tilde\Gamma_n=\{1\}$. Therefore, the left-hand side in (\ref{BB1}) is already in its reduced form so it cannot be trivial
 since it has alternating word length at least $2$. Thus $k_1k_2^{-1}=1$ which means $k_1=k_2$ and $m_1w_1=m_2w_2$ so that $w_2^{-1}w_1=m_1^{-1}m_2\in \tilde\Gamma_n$. Moreover, observe that if $w_1, w_2\in\mathcal{F}_0$, then clearly $w_2^{-1}w_1\in \tilde\Gamma_n$

 From above, on the set $\mathcal{F}$ we can introduce the following equivalence relation:  
 \begin{equation*}
 w_1\sim w_2 \quad\text{if there exists an}\,\,\, m\in \tilde\Gamma_n \,\,\,\text{such that}\,\,\, mw_1=w_2.
 \end{equation*} 
Next let $\mathcal{G}$ be a transversal set for 
$\mathcal F\big/\sim$
, i.e., pick an element $w$ in each equivalence class of $\mathcal F/\sim$. Note that $T_0(\Gamma)=\sqcup_{w\in \mathcal{G}} \tilde\Gamma_n w\Sigma_n'$ is the double coset decomposition. Thus as $Q_n$-$M_n$ bimodules we have the following decomposition:
\begin{equation}\label{BB2}
{}_{Q_n}L^2(M)_{M_n}\cong \bigoplus_{w\in \mathcal{G}}\overline{Q_mu_w M_n}^{\|\cdot\|_2}.
\end{equation}
Next let $\mathcal K$ be a right cosets representatives for the inclusion $\tilde \Gamma_n <T_0(\Gamma)$. Thus as $Q_n$-$M_n$ bimodules we have that
\begin{align}
{}_{Q_n}L^2(M)\bar\otimes L^2(M)_{M_n}
&\cong\bigoplus_{k \in \mathcal K, w\in \mathcal F}\overline{Q_m (u_k\otimes u_w) M_n}^{\|\cdot\|_2}\nonumber \\
&\cong\bigoplus_{k\in \mathcal K, w\in\mathcal{G}}\left (\bigoplus_{\delta\sim w}\overline{Q_m (u_k \otimes u_{\delta}) M_n}^{\|\cdot\|_2}\right )\label{BB3}
\end{align}

Next we argue that argue that for all $w\in \mathcal{G}$, $\delta\in \mathcal F$ and $k\in\mathcal K$  we have that 
\begin{equation}\label{coarsebim}
{}_{Q_n}\overline{Q_nu_wM_n}^{\|\cdot\|_2}_{M_n}\cong {}_{Q_n} L^2(Q_n)\otimes L^2(M_n)_{M_n}\cong {}_{Q_n}\overline{Q_m (u_k \otimes u_{\delta}) M_n}^{\|\cdot\|_2}_{M_n}
\end{equation}
as $Q_n$-$M_n$-bimodules.

To see the first part of \eqref{coarsebim} fix $q_1,q_2\in Q_n$ and $n_1,n_2\in M_n$ and notice that 
\begin{align*}
\langle q_1u_wn_1,q_2u_wn_2 \rangle
&=\tau(q_1u_wn_1n_2^*u_{w^{-1}}q_2^*)\\
&=\tau(q_1u_wE_{Q_n}(n_1n_2^*)u_{w^{-1}}q_2^*)\\
&=\tau(n_1n_2^*)\tau(q_1u_wu_{w^{-1}}q_2^*)\\
&=\tau(n_1n_2^*)\tau(q_1q_2^*)\\
&=\langle q_1\otimes n_1,q_2\otimes n_2  \rangle
\end{align*}
This computation shows that the map $q u_w  n\mapsto q\otimes n$ induces an $Q_n$-$M_n$-bimodules isomorphism  
between $\overline{Q_mu_wM_n}^{\|\cdot\|_2}$ and  $L^2(Q_n)\bar\otimes L^2(M_n)$.

The second part of \eqref{coarsebim} follows in a similar manner as the map $q u_k\otimes u_\delta  n\mapsto q\otimes n$ does the job. Indeed fixing $q_1,q_2\in Q_n$ and $n_1,n_2\in M_n$ we see that 
\begin{align*}
\langle q_1 (u_k\otimes u_\delta) n_1,q_2 (u_k\otimes u_\delta) n_2\rangle
&=\langle q_1 u_k,q_2 u_k\rangle \langle u_\delta n_1, u_\delta n_2\rangle\\
&=\langle q_1,q_2\rangle_{L^2(Q_n)} \langle  n_1,n_2\rangle_{L^2(M_n)}\\
&=\langle q_1\otimes n_1, q_2\otimes n_2\rangle_{L^2(Q_n)\otimes L^2(M_n)}.
\end{align*}
Now combining  relations \eqref{BB2}, \eqref{coarsebim} and \eqref{BB3} we see that, as $Q_n$-$M_n$ bimodules we have the following 
\begin{align*}
{}_{Q_n}L^2(M)_{M_n}
&\cong \bigoplus_{w\in \mathcal{G}}{}_{Q_n} \overline{Q_nu_wM_n}_{M_n}^{\|\cdot\|_2}\\
&\cong \bigoplus_{w\in \mathcal{G}}{}_{Q_n} L^2(Q_n)\bar\otimes L^2(M_n) _{M_n}\\
&\prec \bigoplus_{k\in \mathcal K, w\in \mathcal{G}}\big(\bigoplus_{\delta\sim w}{}_{Q_n} L^2(Q_n)\bar\otimes L^2(M_n) _{M_n}\big)\\
&\cong \bigoplus_{k\in\mathcal K,w\in \mathcal{G}}\big(\bigoplus_{\delta\sim w}{}_{Q_n}\overline{Q_n (u_k\otimes u_\delta) M_n}^{\|\cdot\|_2}_{M_n}\big)\\
&\cong {}_{Q_n} L^2(Q_n)\bar\otimes L^2(M_n) _{M_n}.
\end{align*}
This concludes the proof of (\ref{eqq2}).

Notice that relations (\ref{eqq1}) and (\ref{eqq2}) show that the conditions in Lemma \ref{lemma2.6IS19} are satisfied. 
Since {$P_i\prec^s_M L(\tilde\Gamma_n)$} by (\ref{BB6}) then we have that $P_i$ is amenable relative to $\cap_nQ_n=\mathbb{C}1.$ Thus, $P_i$ is amenable and we are done.
In the case $\alpha=1$ and can let $\Sigma_n'=(\bigoplus_{i=1}^n)\rtimes \mathfrak S_n\vee\Lambda_1\vee\Lambda_2\cdots\vee\Lambda_n$ and the same method above applies verbatim. \end{proof}

Notice that the previous theorem can be generalized by to the case of products $\Omega= \Omega_1\times ...\times \Omega_n$ of McDuff's groups $\Omega_i = T_{\alpha_i}(\Gamma)$. Specifically it asserts that all possible tensor splittings $L(\Omega)=P_1\bar\otimes P_2$ occurs only in the ``amenable rooms'' around the subproducts $\Gamma$ of $\Gamma$. The proof follows essentially the same arguments as in the proof Theorem \ref{MDsplitting} and is left to the reader. 

\begin{thm}
For $n\geq 2$ and $i\in\{1,\ldots,n\}$, fix $\Gamma_i$  non-amenable groups. Let $\alpha_i\in \{0,1\}$ and let $\Omega_i=T_{\alpha_i}(\Gamma_i)$. Denote by $\Omega=\Omega_1\times \Omega_2\times \cdots\times \Omega_n$ and assume that $M=L(\Omega)=P_1\bar\otimes P_2$ where $P_i$ are non-amenable factors. Then there exist $i\in\{1,2\}$ and a subset $I \subsetneq \{1,2,\ldots,n\}$ such that $P_i$ is amenable relative to $L(\Omega_I)$ inside $M$.
\end{thm}


%
%

\biblio{thesis}

\nocite{*}
	\bibliographystyle{amsalpha}


\end{document}